\documentclass[11pt,a4paper,reqno]{amsart}
\usepackage[T1]{fontenc}
\usepackage{lmodern}
\usepackage{amsmath,amssymb,amsthm,mathtools}
\usepackage[margin=30mm]{geometry}
\usepackage{microtype}
\usepackage{enumitem}
\usepackage[hidelinks]{hyperref}
\usepackage{bookmark}
\usepackage{booktabs}
\usepackage{graphicx} 
\usepackage{url}
\usepackage{amsmath,amsfonts,amssymb,amsthm}
\usepackage{geometry}
\usepackage{color}
\usepackage{algorithm}
\usepackage{algpseudocode}
\usepackage{comment}
\hypersetup{pdftitle={Automorphic transcendence degree, growth, and skew normalization},
  pdfsubject={Skew polynomial rings and relative growth}}
\numberwithin{equation}{section}
\newtheorem{theorem}{Theorem}[section]
\newtheorem{proposition}[theorem]{Proposition}
\newtheorem{lemma}[theorem]{Lemma}
\newtheorem{corollary}[theorem]{Corollary}
\theoremstyle{definition}
\newtheorem{definition}[theorem]{Definition}
\newtheorem{example}[theorem]{Example}

\theoremstyle{remark}
\newtheorem{remark}[theorem]{Remark}
\theoremstyle{plain}
\newtheorem*{theoremA}{Theorem A}
\newtheorem*{theoremB}{Theorem B}
\newtheorem*{theoremC}{Theorem C}
\newtheorem*{theoremD}{Theorem D}
\DeclareMathOperator{\atrdeg}{atrdeg}
\DeclareMathOperator{\Aut}{Aut}
\DeclareMathOperator{\Inn}{Inn}
\DeclareMathOperator{\Out}{Out}
\DeclareMathOperator{\Supp}{supp}
\DeclareMathOperator{\LE}{LE}
\DeclareMathOperator{\lexp}{le}
\DeclareMathOperator{\GKdim}{GKdim}
\DeclareMathOperator{\rgdim}{rgdim}

\DeclareMathOperator{\Span}{span}
\DeclareMathOperator{\Ad}{Ad}
\newcommand{\N}{\mathbb N_0}
\newcommand{\Z}{\mathbb Z}
\newcommand{\Q}{\mathbb Q}

\newcommand{\id}{\operatorname{id}}
\newcommand{\bx}{\overline{x}}

\newcommand{\dL}{\mathcal{L}_D}
\newcommand{\dR}{\mathcal R_D}
\newcommand{\ldim}[1]{\dim({}_D #1)}
\newcommand{\rdim}[1]{\dim(#1{}_D)}
\newcommand{\doi}[1]{\href{https://doi.org/#1}{\nolinkurl{#1}}}
\newcommand{\arxiv}[1]{\href{https://arxiv.org/abs/#1}{arXiv:\nolinkurl{#1}}}
\setlist[enumerate,1]{label=\textup{(\roman*)},leftmargin=2.3em,itemsep=2pt}
\allowdisplaybreaks[2]

\title[Automorphic transcendence degree]{Automorphic transcendence degree, growth, and skew normalization}

\author{ Dinh Van Hoang}
\address[Dinh Van Hoang]{Faculty of Advanced Education, Ho Chi Minh City University of Technology and Engineering, Vietnam }
\email{hoangdv@hcmute.edu.vn}

\author{Vo Ngoc Thieu}
 \address[Vo Ngoc Thieu]{Department of Computer Science, University of Bath, United Kingdom}
\email{ntv22@bath.ac.uk}

\author{Phan Thanh Toan}
\address[Phan Thanh Toan]{Analytical and Algebraic Methods in Optimization Research Group, Faculty of Mathematics and Statistics, Ton Duc Thang University, Ho Chi Minh City, Vietnam}
\email{phanthanhtoan@tdtu.edu.vn}

\date{}

\begin{document}
\begin{abstract}
We study automorphic transcendence degree for quotients of multivariate skew polynomial rings over a division ring \(D\). We identify this invariant with relative growth dimension, the degree of the relative Hilbert polynomial, and a coordinate dimension determined by leading monomials. We establish invariance under finite module extensions, integral extensions, and investigate the Gelfand-Kirillov dimension under a local finiteness condition. When the coefficient automorphisms are independent modulo inner automorphisms, we prove that a quotient is automorphically normalizable precisely when its automorphic transcendence degree equals the number of nonnilpotent coordinate variables, and classify all such normalization subrings.
\end{abstract}
\maketitle
\markboth{AUTOMORPHIC TRANSCENDENCE DEGREE}{AUTOMORPHIC TRANSCENDENCE DEGREE}
\tableofcontents
\section{Introduction}\label{sec:intro}

Noether normalization connects algebraic independence with the structure
and dimension of a finitely generated commutative algebra. If $A$ is a
nonzero finitely generated algebra over a field $k$, there exist
algebraically independent elements $y_1,\ldots,y_d\in A$ such that $A$
is finite as a module over $k[y_1,\ldots,y_d]$. The number $d$ is the
Krull dimension of $A$; it is also the largest size of an algebraically
independent family in $A$ and the degree of its cumulative Hilbert
polynomial. Thus independence, growth, and finite generation over a
polynomial subring give complementary descriptions of the same
dimension. Over a division ring, this relationship raises two closely
related questions: which polynomial subrings should measure
transcendence, and when does a subring of the resulting dimension
provide a finite normalization?

Recent work on skew Noether normalization gives a natural setting for
these questions. Paran and Vo~\cite{ParanVo2025} proved normalization
for quotients of polynomial rings in central variables over a division
ring, established sufficient conditions in the presence of commuting
coefficient automorphisms, and exhibited skew polynomial quotients for
which normalization fails. Hoang and Toan~\cite{HoangToan2026}
extended the study to multivariate skew polynomial rings with derivations.
These results motivate a dimension theory that remains meaningful
even when a finite normalization does not exist. In this paper, we
develop such a theory through automorphic transcendence degree and
use it to characterize and classify normalizations for a class of
skew polynomial quotients.

Let $D$ be a division ring of characteristic zero. Our main results
concern extensions with a presentation
\begin{equation}\label{eq:intro-presentation_0}
 S=P/I,\qquad
 P=D[x_1,\ldots,x_m;(\sigma_1,\delta_1),\ldots,
                         (\sigma_m,\delta_m)],
\end{equation}
where $I$ is a proper two-sided ideal, the variables commute pairwise,
and
\[
 x_i a=\sigma_i(a)x_i+\delta_i(a)\qquad(a\in D).
\]
Here $\sigma_i\in\Aut(D)$, each $\delta_i$ is a
$\sigma_i$-derivation, and the coefficient maps satisfy
\begin{equation}\label{eq:intro-compatibility}
 \sigma_i\sigma_j=\sigma_j\sigma_i,\qquad
 \sigma_i\delta_j=\delta_j\sigma_i,\qquad
 \delta_i\delta_j=\delta_j\delta_i
 \qquad(i\ne j).
\end{equation}
These compatibility conditions ensure that the ordered monomials
$x^\alpha$ form a basis over $D$. We call $S/D$ a
\emph{compatible Ore quotient}, or an \emph{automorphically finitely
generated extension}. The pairwise commutativity of the variables
is part of this hypothesis; arbitrary iterated Ore extensions can
have different independence properties.

For any unital extension $D\subseteq S$, an element $s\in S$ is
\emph{automorphic over $D$} if
\begin{equation}\label{eq:intro-action_0}
 sa=\sigma(a)s+\delta(a)\qquad(a\in D)
\end{equation}
for some $\sigma\in\Aut(D)$ and some $\sigma$-derivation $\delta$.
A tuple $s_1,\ldots,s_r$ is \emph{automorphically independent over
$D$} if its elements commute, each is automorphic, and its indexed
monomials $(s_1^{\alpha_1}\cdots s_r^{\alpha_r})_{\alpha\in\N^r}$
are left $D$-linearly independent. We denote the supremum of the
lengths of these tuples by $\atrdeg_D(S)$, allowing the value
$\infty$. We prove that left and right independence are equivalent
and that an independent tuple generates a compatible skew polynomial
subring. A \emph{skew normalization} of $S/D$ is  subring generated by a tuple of automorphically independent elements, over
which $S$ is finite as a left module. Both automorphic elements and
normalizations are understood here to allow derivation terms, as in
\cite[Definition~3.5]{HoangToan2026}.

To relate this invariant to growth, let $F_qS$ be the image in $S$ of
the polynomials of total degree at most $q$, and set
\[
 h_S(q)=\ldim{F_qS},\qquad
 \rgdim_D(S)=\limsup_{q\to\infty}\frac{\log h_S(q)}{\log q}.
\]
Fix a degree-compatible monomial
order, let $\LE(I)$ be the set of leading exponents of nonzero
elements of $I$, and put $\Delta=\N^m\setminus\LE(I)$. For
$J\subseteq\{1,\ldots,m\}$, write $\N^J$ for the exponent vectors
supported in $J$. Our first main result is the following theorem.

\begin{theoremA}
Let $S=P/I$ be a compatible Ore quotient as in
\eqref{eq:intro-presentation_0}. Then
\[
 d:=\atrdeg_D(S)
   =\rgdim_D(S)
   =\max\{|J|:\N^J\subseteq\Delta\}.
\]
The function $h_S(q)$ is eventually a polynomial of degree $d$ with
positive leading coefficient, and
\[
 h_S(q)=\Theta((q+1)^d),\qquad
 \lim_{q\to\infty}\frac{\log h_S(q)}{\log q}=d.
\]
Moreover, $d$ is the largest length of any commuting tuple whose
indexed monomials are left $D$-linearly independent,  without
requiring the elements to be automorphic. A subtuple of the coordinate
images $\bx_1,\ldots,\bx_m$ attains this maximum.
\end{theoremA}

The proof combines standard monomial bases with polynomial growth
bounds. Leading exponents determine the filtered dimensions, while
an independent commuting $r$-tuple contributes at least
$\binom{q+r}{r}$ independent monomials in filtration degree at most
a constant multiple of $q$. The use of leading monomials to study
Hilbert functions and growth is part of the established Gr\"obner
theory of noncommutative polynomial rings~\cite{Bueso2003}; related
computations for differential difference algebras and their modules
appear in~\cite{ZhaoZhang2016}. Here this approach identifies the
independence invariant, gives an explicit rational expression for
the cumulative Hilbert series, and computes the dimension from the
supports of the minimal leading exponents.

This comparison also places automorphic transcendence degree in
relation to other noncommutative dimension theories. Zhang studied
Gelfand--Kirillov transcendence degree and lower transcendence
degree~\cite{Zhang1996,Zhang1998}, and Bell introduced strong lower
transcendence degree for division algebras~\cite{Bell2012}. Our
invariant records commuting polynomial families relative to a fixed,
possibly noncentral, division subring. On the growth side, Lezama
and Latorre developed generalized Hilbert functions and
Gelfand--Kirillov dimension for semi-graded rings~\cite{Lezama2017}.
We use an explicit relative convention: $\GKdim_D(S)$ is the
supremum of
\[
 \limsup_{q\to\infty}
 \frac{\log\ldim{V^q}}{\log q},\qquad
 V^q=\Span_D\{v_1\cdots v_q:v_i\in V\},
\]
over finite-dimensional left $D$-subspaces $V\subseteq S$ containing
$1$. When $S$ satisfies finiteness conditions, we prove the following result.
\begin{theoremB}
Let $D\subseteq R\subseteq S$, where $R/D$ is automorphically
finitely generated and $S$ is finite as a left $R$-module. Assume, in addition, that $S$ is locally finite as a $D$-bimodule,
in the following left-dimensional sense:
\begin{equation}
 \ldim{DsD}<\infty\qquad\text{for every }s\in S.
\end{equation}
Here $DsD$ denotes the additive span of the elements $asb$, with
$a,b\in D$. Then every power of every left $D$-frame in $S$ is
finite-dimensional, and
\[
 \GKdim_D(S)=\GKdim_D(R)=\atrdeg_D(R)=\atrdeg_D(S).
\]
\end{theoremB}
Thus the independence dimension agrees with relative
Gelfand--Kirillov growth under this convention.

The growth comparison yields invariance results beyond the class of
compatible Ore quotients.

\begin{theoremC}
Let $D\subseteq R\subseteq S$, where $R/D$ is automorphically
finitely generated. If $S$ is finite as a left $R$-module, or if every
element of $S$ satisfies a monic left polynomial relation over $R$,
then
\[
 \atrdeg_D(S)=\atrdeg_D(R).
\]
This common value is also the largest length of a commuting tuple
in $S$ with left $D$-linearly independent indexed monomials. No
automorphic finite-generation assumption on $S/D$ is required.
\end{theoremC}

In particular, every skew normalization of a fixed extension has
exactly $\atrdeg_D(S)$ variables. We also prove that, for a compatible
Ore quotient $S/D$, every automorphically independent tuple of
maximum size makes $S$ left algebraic over the subring it generates.
This conclusion is effective: if $n=\atrdeg_D(S)$ and
$s_1,\ldots,s_n,u\in F_cS$, then
\[
 \binom{q+n+1}{n+1}>h_S(cq)
\]
guarantees a nonzero relation $\sum_{j=0}^q b_j u^j=0$ with
$b_j\in D[s_1,\ldots,s_n]$ and $\deg_s b_j\leq q-j$.
Such a relation need not be monic. Indeed,
Example~\ref{ex:algebraic-not-stable} shows that elementwise left
algebraicity alone does not preserve automorphic transcendence degree.

Our final main result concerns the existence and structure of
normalizations. Specialize to zero derivations in the presentation,
so that
$P=D[x_1,\ldots,x_m;\sigma_1,\ldots,\sigma_m]$, and suppose that
the commuting coefficient automorphisms are independent modulo inner
automorphisms. More precisely, require the homomorphism
\begin{equation}\label{eq:Phi-intro}
 \Phi:\Z^m\longrightarrow\Out(D)=\Aut(D)/\Inn(D),\qquad
 \alpha\longmapsto
 [\sigma_1^{\alpha_1}\cdots\sigma_m^{\alpha_m}]
\end{equation}
to be injective. We call this condition \emph{outer independence}.

\begin{theoremD}
Assume outer independence. Every two-sided ideal of $P$ is monomial.
Let $S=P/I$ be a proper quotient and set
\[
 E=\{i:\bx_i\text{ is not nilpotent in }S\}.
\]
Every automorphic element of $S\setminus D$ has the form
$b+c\overline{x^\alpha}$, where $b\in D$, $c\in D^\times$, and
$x^\alpha$ is a nonconstant standard monomial. Furthermore,
\[
 S\text{ admits a skew normalization over }D
 \quad\Longleftrightarrow\quad
 \atrdeg_D(S)=|E|.
\]
When these conditions hold, the normalization subrings are precisely
\[
 D[\bx_i^{a_i}:i\in E],\qquad a_i\geq1.
\]
If $\bx_i^{N_i}=0$ for $i\notin E$, then $S$ has a left generating
set over this subring of size at most
\[
 \Bigl(\prod_{i\in E}a_i\Bigr)
 \Bigl(\prod_{i\notin E}N_i\Bigr).
\]
\end{theoremD}

Outer independence separates monomial components and restricts
automorphic elements to the form stated in Theorem~D. Finite module generation then forces a normalization tuple
to account for every nonnilpotent coordinate axis. Theorem~D gives a criterion
for every proper quotient under outer independence and determines
all normalization subrings. 

The paper is organized as follows. Section~\ref{sec:prelim} develops
automorphic independence, coefficient compatibility, and left--right
symmetry. Section~\ref{sec:growth} establishes the growth comparison
and the Hilbert polynomial and series formulas.
Section~\ref{sec:extensions} proves effective algebraicity, extension
invariance, and the relative Gelfand--Kirillov comparisons.
Section~\ref{sec:outer} proves the normalization criterion and
classification under outer independence. Finally,
Section~\ref{sec:examples} gives the computation from leading supports
and examples illustrating the distinction between dimension and
normalization.

\section{Compatible Ore extensions and automorphic independence}
\label{sec:prelim}

Throughout this paper $D$ is a division ring of characteristic 0, all rings are associative with identity, and all inclusions and
homomorphisms preserve identity. We write
$\N=\{0,1,2,\ldots\}$ and use $\N^0=\{()\}$.  
 Unless otherwise stated,
dimensions over $D$ are left dimensions. The notation $D[T]$ denotes
the subring generated by $D$ and a subset $T$; it imposes no unstated
commutation relations with $D$. The material for this section can be found in \cite{GoodearlWarfield2004, HoangToan2026,  ParanVo2025}.

\begin{definition}[Compatibility of the coefficient maps]
     For each automorphism $\sigma\in\Aut(D)$, a $\sigma$-derivation of $D$ is an additive map
$\delta:D\to D$ satisfying
\[
 \delta(ab)=\sigma(a)\delta(b)+\delta(a)b\qquad(a,b\in D).
\]
In particular, $\delta(1)=0$. A family
$((\sigma_i,\delta_i))_{i=1}^m$ is called \emph{compatible} if
\begin{equation}\label{eq:compatibility}
 \begin{aligned}
 \sigma_i\sigma_j&=\sigma_j\sigma_i,\\
 \sigma_i\delta_j&=\delta_j\sigma_i,\\
 \delta_i\delta_j&=\delta_j\delta_i
 \end{aligned}
 \qquad\text{for all distinct }i,j.
\end{equation}
The mixed identity is imposed for every ordered pair of distinct
indices; no condition
$\sigma_i\delta_i=\delta_i\sigma_i$ is imposed.
\end{definition}
.
\begin{definition} \label{def:automorphic}
Let $D\subseteq S$.
    An element $s\in S$ is
\emph{automorphic over $D$} if
\begin{equation}\label{eq:intro-action}
 sd=\sigma(d)s+\delta(d)\qquad(d\in D).
\end{equation}
for some $\sigma\in\Aut(D)$ and some $\sigma$-derivation $\delta$. We call $(\sigma,\delta)$ an associated pair of $s$.
\end{definition}




\begin{lemma}[Uniqueness and compatibility]\label{lem:compatibility}
Let $D\subseteq S$ be a ring extension and $s,t$ be elements of $S$. 
\begin{enumerate}
\item If $s\notin D$ is automorphic, its associated pair is unique.
\item Suppose $s,t$ commute and have associated pairs
$(\sigma,\delta)$ and $(\tau,\varepsilon)$. If $1,s,t,st$ are left
$D$-linearly independent, then
\[
 \sigma\tau=\tau\sigma,\quad
 \sigma\varepsilon=\varepsilon\sigma,\quad
 \delta\tau=\tau\delta,\quad
 \delta\varepsilon=\varepsilon\delta.
\]
\end{enumerate}
\end{lemma}

\begin{proof}
If $s$ has two associated pairs, subtraction gives
\[
 (\sigma(d)-\tau(d))s=\varepsilon(d)-\delta(d)\in D.
\]
A nonzero coefficient on the left is invertible and would imply
$s\in D$. Thus $\sigma=\tau$ and then $\delta=\varepsilon$.
For the second assertion, expand $s(td)=t(sd)$ and compare the coefficients of
$st,s,t,1$.
\end{proof}

\begin{remark}\label{rem:uniqueness-warning}
The condition $s\notin D$ cannot be replaced by $s\ne0$. For any
$c\in D$ and $\sigma\in\Aut(D)$, the map
$\delta(d)=cd-\sigma(d)c$ is a $\sigma$-derivation and makes $c$
automorphic with associated pair $(\sigma,\delta)$.
\end{remark}

\begin{lemma}[Triangular coefficient movement]\label{lem:triangular}
Suppose $s_1,\ldots,s_r$ commute and each has an associated pair
$(\sigma_i,\delta_i)$. Then:
\begin{itemize}
    \item [(i)] For $\alpha\in\N^r$, set
$\sigma^\alpha=\sigma_1^{\alpha_1}\cdots\sigma_r^{\alpha_r}$ in the
indicated order. For every $d\in D$ there are coefficients in $D$ such
that
\begin{align}
 s^\alpha d
 &=\sigma^\alpha(d)s^\alpha
   +\sum_{\substack{\beta\le\alpha\\|\beta|<|\alpha|}}
          c_{\alpha,\beta}(d)s^\beta,
 \label{eq:move-left}\\
 d s^\alpha
 &=s^\alpha(\sigma^\alpha)^{-1}(d)
   +\sum_{\substack{\beta\le\alpha\\|\beta|<|\alpha|}}
          s^\beta e_{\alpha,\beta}(d).
 \label{eq:move-right}
\end{align}
Consequently, such a tuple is left independent if and only if it is
right independent.
\item[(ii)] If $s_1,\dots,s_r$ are left algebraically independent over $D$, then the family of associated pairs $(\sigma_1,\delta_1),\dots,(\sigma_r,\delta_r)$ satisfies the compatible condition (\ref{eq:compatibility}).
\end{itemize}

\end{lemma}

\begin{proof}
The degree-one rules are the defining relation and
\begin{equation}\label{eq:reverse-rule}
 d s_i=s_i\sigma_i^{-1}(d)
        -\delta_i\bigl(\sigma_i^{-1}(d)\bigr).
\end{equation}
Repeatedly apply these rules through the fixed ordered product
$s_1^{\alpha_1}\cdots s_r^{\alpha_r}$. At each step the automorphism
term preserves the exponent of the variable being crossed, while
the derivation term lowers it by one. This proves
\eqref{eq:move-left} and \eqref{eq:move-right} by induction. The inverse
in \eqref{eq:move-right} has the reverse composition order, as required.

Suppose a nontrivial right relation $\sum_\alpha s^\alpha d_\alpha=0$
is given. Convert it to left coefficients using
\eqref{eq:move-left}. Among indices with nonzero coefficient, choose
one of maximum total degree. Its new coefficient is
$\sigma^\alpha(d_\alpha)\ne0$; corrections from all terms have strictly
lower degree, so the converted relation is nontrivial. A left
independent tuple therefore has no right relation. The converse uses
\eqref{eq:move-right} in exactly the same way. 
Since $s_1,\dots,s_r$ are left algebraically independent over $D$, it follows from Lemma \ref{lem:compatibility} (ii) that $(\sigma_1,\delta_1),\dots,(\sigma_r,\delta_r)$ satisfy the compatible condition.
\end{proof}

\begin{definition}[Automorphic transcendence degree]\label{def:independence}
An element $s\in S$ is \emph{automorphic over $D$ with associated pair}
$(\sigma,\delta)$ if it satisfies \eqref{eq:intro-action}. A tuple
$s_1,\ldots,s_r$ is \emph{automorphically left algebraically independent over $D$} if
its members commute, each is automorphic over $D$, and the indexed
family
\[
 (s^\alpha)_{\alpha\in\N^r},\qquad
 s^\alpha=s_1^{\alpha_1}\cdots s_r^{\alpha_r},
\]
is left $D$-linearly independent.  We often omit ``algebraically'' from
``automorphically algebraically independent'' for brevity.

Right independence is defined by
putting the coefficients on the right.
 Define the \emph{left automorphic transcendence degree} of $S$ over $D$ as
\[
 \atrdeg_D^l(S)=\sup\{n\in\N:
   S\text{ contains an automorphically left independent }n\text{-tuple}\},
\]
and define $\atrdeg_D^r(S)$ analogously. Both suprema take values in
$\N\cup\{\infty\}$. The empty tuple is allowed, with its single
monomial $1$.
\end{definition}

By Lemma \ref{lem:triangular}, we henceforth write $\atrdeg_D(S)$ for the common value. The invariant
is monotone under inclusions fixing $D$ and is preserved by isomorphisms
fixing $D$.
maximal under inclusion.

\begin{definition}\label{def:commuting-degree}
For an arbitrary extension $S/D$, define $\dL(S)$ to be the supremum of the
lengths of commuting tuples whose indexed monomials are left
$D$-linearly independent:
$$\dL(S):=\sup\{n\in\N:
   S\text{ contains a commuting left independent }n\text{-tuple}\}$$. 
Define $\dR(S)$ using right independence. Then
\[
 \atrdeg_D(S)\le\dL(S),\qquad
 \atrdeg_D(S)\le\dR(S).
\]
No equality between these three invariants is assumed for an arbitrary
extension.
\end{definition}

{
\begin{proposition}[Construction of the polynomial ring]
\label{prop:construction}
For a compatible  family $\big\{(\sigma_i,\delta_i)_{i=1,\dots,m}\big\}$, there is a ring
\[
 P=D[x_1,\ldots,x_m;(\sigma_1,\delta_1),\ldots,(\sigma_m,\delta_m)]
\]
with left $D$-basis $(x^\alpha)_{\alpha\in\N^m}$ and relations
\begin{equation}\label{eq:relations}
 x_ix_j=x_jx_i,\qquad
 x_i d=\sigma_i(d)x_i+\delta_i(d)\quad(d\in D).
\end{equation}
Conversely, if
\eqref{eq:relations} holds in a ring with this monomial basis, the
coefficient family $\big\{(\sigma_i,\delta_i)_{i=1,\dots,m}\big\}$ satisfies \eqref{eq:compatibility}.
\end{proposition}}

\begin{proof}
Proceed by induction on $m$. On the ring generated by
$D,x_1,\ldots,x_{i-1}$, extend $\sigma_i$ by fixing each earlier variable,
and extend $\delta_i$ by annihilating each earlier variable. These maps
respect the coefficient relations: for $j<i$, the identities needed
for the extended automorphism are
$\sigma_i\sigma_j=\sigma_j\sigma_i$ and
$\sigma_i\delta_j=\delta_j\sigma_i$. For the extended derivation, applying
it to $x_jd=\sigma_j(d)x_j+\delta_j(d)$ gives the requirement
\[
 x_j\delta_i(d)
 =\delta_i\sigma_j(d)x_j+\delta_i\delta_j(d),
\]
which follows from the remaining cross-identities. The relations
between earlier variables are preserved because the extended
automorphism fixes them and the extended derivation annihilates them.
Adjoining $x_i$ by the ordinary one-variable Ore
construction therefore gives the desired ring. Iterating its free
left-module description gives the stated basis. 
The converse follows from Lemma \ref{lem:triangular} (ii).
\end{proof}

\begin{definition}\label{def:afg}
An extension $S/D$ is a \emph{compatible Ore quotient}, or is
\emph{automorphically finitely generated}, if it has a presentation
\begin{equation}\label{eq:intro-presentation}
 S=P/I,\qquad
 P=D[x_1,\ldots,x_m;(\sigma_1,\delta_1),\ldots,(\sigma_m,\delta_m)],
\end{equation}
where $I$ is a proper two-sided ideal, $\sigma_i\in\Aut(D)$, each $\delta_i$ is a $\sigma_i$-derivation, the $x_i$ commute with each other, and the
coefficient family $(\sigma_i,\delta_i)$ satisfies the compatible condition (\ref{eq:compatibility}), and 
$$x_i d=\sigma_i(d)x_i+\delta_i(d),\quad \forall d\in D.$$
\end{definition}

\begin{lemma}\cite[Prop 2.9]{HoangToan2026}\label{lem:subring}
Let $S/D$ be a ring extension. Let $s_1,\ldots,s_r\in S$ be automorphically independent with associated
pairs $(\sigma_i,\delta_i)$. Substitution induces an isomorphism fixing
$D$,
\[
 D[T_1,\ldots,T_r;(\sigma_1,\delta_1),\ldots,(\sigma_r,\delta_r)]
 \xrightarrow{\ \sim\ }D[s_1,\ldots,s_r],
 \qquad T_i\longmapsto s_i.
\]
The monomials $s^\alpha$ form both a left and a right $D$-basis of the
subring. The filtrations obtained from the two coefficient conventions
coincide.
\end{lemma}

\begin{proof}
The associated pairs are unique because each $s_i\notin D$.
Lemma~\ref{lem:compatibility} gives their compatibility. The construction
in Proposition~\ref{prop:construction} and the defining relations of the
$s_i$ give the substitution homomorphism. Its image is the generated
subring, and its kernel is zero by left independence. Moving coefficients
by Lemma~\ref{lem:triangular} gives right spanning and right independence;
the lower-degree corrections also show equality of the degree
filtrations.
\end{proof}

\begin{definition}\label{def:normalization}
An extension $S/D$ is \emph{skew normalizable on the left} if it is
finite as a left module over $D[s_1,\ldots,s_r]$ for an automorphically
independent tuple. The subring is called a \emph{left skew normalization
subring}. The empty tuple is allowed. Unless a side is specified,
``skew normalization'' means left skew normalization. Right skew
normalization is defined analogously.
\end{definition}

The coefficient associated pairs $(\sigma,\delta)$ in a normalization are not required to be the
pairs in a specified presentation of $S$. This is the derivation-inclusive
notion in \cite[Definition~3.5]{HoangToan2026}, not the stricter
automorphism-only notion in \cite[Definition~2.16]{ParanVo2025}.

\section{Standard monomials, relative growth, and Hilbert functions}
\label{sec:growth}

Let $D$ be a division ring. Fix an automorphic presentation $S=P/I$ with $m$ commuting variables $x_1,\dots,x_m$, where  $P=D[x_1,\ldots,x_m;(\sigma_1,\delta_1),\ldots,(\sigma_m,\delta_m)]$  and $I$ is a proper two-sided ideal of $P$. Define
\[
 F_qP=\bigoplus_{|\alpha|\leq q}Dx^\alpha,\qquad
 F_qS=(F_qP+I)/I\qquad(q\in\N).
\]
This is an exhaustive filtration by left finite-dimensional $D$-vector spaces,
with $F_0S=D$ and 
\begin{equation}\label{inclusion1}
    F_pS\,F_qS\subseteq F_{p+q}S.
\end{equation}
We define the Hilbert function of $S$ with respect to the filtration $F$ as $$h_S(q)=\ldim{F_qS},\quad q\in \N.$$

Choose a degree-compatible monomial order $\prec$ on $\N^m$:
it is a well-order compatible with addition, and $|\alpha|<|\beta|$
implies $\alpha\prec\beta$. For example, one can choose $\prec$ to be the graded lexicographic order. For $0\ne f\in P$, let $\lexp(f)$ be its
largest exponent with respect to $\prec$. Put
\[
 \Lambda=\LE(I)=\{\lexp(f):0\ne f\in I\},\qquad
 \Delta=\N^m\setminus\Lambda.
\]
When $I=0$ we take $\Lambda=\varnothing$.

\begin{lemma}[Monomial $D$-basis]\label{lem:standard}
\begin{enumerate}
    \item The set $\Lambda$ is upward closed for the componentwise order: if $\alpha\in \Lambda$, then $\alpha+\beta\in\Lambda$ for all $\beta\in\N^{m}$.
    \item The classes $$(\overline{x^\alpha}:=x^\alpha+I)_{\alpha\in\Delta}$$ form a left
$D$-basis of $S$. 
    \item For each $q\in\N$, the set $\big\{\overline{x^\alpha}:{\alpha\in\Delta},\ |\alpha|\leq q\big\}$  forms a basis of $F_qS$, and consequently
\begin{equation}\label{eq:hilbert-count}
 h_S(q)=\rdim{F_qS}
 =\#\{\alpha\in\Delta:|\alpha|\leq q\}.
\end{equation}
\end{enumerate}
\end{lemma}

\begin{proof}
(i) If $\alpha\in\Lambda$, then $\alpha=\lexp(f)$ for $0\ne f\in I$. For any $\beta\in\N^{m}$, 
$\lexp(x^\beta f)=\alpha+\beta$ by
Lemma \ref{lem:triangular}, and $x^\beta f\in I$.
Thus $\Lambda$ is upward closed.

Now we prove (ii) and (iii). Take any monomial \(x^\alpha\). If \(\alpha \in \Delta\), there is nothing to do. If \(\alpha \in \Lambda\), by definition there exists \(f \in I\) whose leading exponent is \(\alpha\). After multiplying \(f\) on the left by a suitable nonzero element of \(D\), we may write
\[
f
=
x^\alpha
+
\sum_{\beta \prec \alpha} c_\beta x^\beta,
\]
where \(\beta \prec \alpha\) refers to the chosen monomial order. 
Since \(f \in I\), in the quotient \(S=P/I\) we have
\[
\overline{x^\alpha}
=
-\sum_{\beta \prec \alpha} c_\beta \overline{x^\beta}.
\]
Thus the class of \(x^\alpha\) is expressed as a combination of classes of strictly smaller monomials.
 If some of those smaller exponents still belong to \(\Lambda\), we repeat the reduction. Because the monomial order is well-ordered, this process must stop after finitely many steps. Eventually we express \(\overline{x^\alpha}\) as a left \(D\)-linear combination of monomials whose exponents lie in \(\Delta\). Hence
the set $\left\{
\overline{x^\alpha} : \alpha \in \Delta
\right\}$
spans \(S\) as a left \(D\)-vector space and the set $\left\{
\overline{x^\alpha} : \alpha \in \Delta,|\alpha|\le q
\right\}$
spans \(F_q S\) as a left \(D\)-vector space.

Next, we show that $\left\{
\overline{x^\alpha} : \alpha \in \Delta
\right\}$ is left independent over $D$. Suppose there were a nontrivial relation
\[
\sum_{\alpha \in U}
d_\alpha \overline{x^\alpha}
=
0,
\]
where \(U \subseteq \Delta\) is finite and not all \(d_\alpha\) vanish. Then we have $g
:=
\sum_{\alpha \in U}
d_\alpha x^\alpha
\in I.$ 
Choose the largest exponent
\[
\alpha_0
=
\max
\left\{
\alpha \in U : d_\alpha \neq 0
\right\}
\]
with respect to the monomial order. Then
$\operatorname{le}(g)=\alpha_0.$ 
Since \(g \in I\), by definition of \(\Lambda\),
\[
\alpha_0 \in \Lambda.
\]
But every exponent occurring in \(g\) was chosen from \(U \subseteq \Delta\), so $\alpha_0 \in \Delta.$
This is impossible because $\Delta = \N^m \setminus \Lambda$.
Therefore no nontrivial relation can exist. Thus
$$\left\{
\overline{x^\alpha} : \alpha \in \Delta
\right\}
\text{ is a left \(D\)-basis of \(S\)},
$$
and $$\left\{
\overline{x^\alpha} : \alpha \in \Delta,\ |\alpha|\le q
\right\}
\text{ is a left \(D\)-basis of \(F_qS\).}
$$
Finally, for any two elements $\beta\prec\alpha$ in $\Delta$, if $\overline{x^\alpha}=\overline{x^\beta}$ then $x^\alpha-x^\beta\in I$, and hence $\mathrm{le}(x^\alpha-x^\beta)=\alpha$. However, $\Lambda$ is upward closed, $\alpha\in \Lambda$. This is a contradiction. Hence $\#\left\{
\overline{x^\alpha} : \alpha \in \Delta
\right\}=|\Delta|$. Therefore, 
$$ h_S(q)=\rdim{F_qS}
 =\#\{\alpha\in\Delta:|\alpha|\leq q\}.$$
This completes our proof.
\end{proof}

\vspace*{0.5cm}
Now, for each coordinate set $J\subseteq\{1,\ldots,m\}$, define
$ \N^J=\{\alpha\in\N^m:\alpha_i=0\text{ for }i\notin J\}.$
Set $$ d_\Delta=\max\{|J|:\N^J\subseteq\Delta\}.$$
The empty subset is admissible because $0\in\Delta$. 
If $\mathbb N^J\subseteq\Delta$, then all monomials in the elements
\[
\{\overline{x_j}=x_j+I:j\in J\}
\]
are left $D$-linearly independent. Since the elements $\overline{x_j}$,
$j\in J$, commute pairwise and are automorphic over $D$, they form an
automorphically left algebraically independent set. Therefore
\begin{equation}\label{inequality_0}
    \operatorname{atrdeg}^{\,l}_D(S)\geq d_\Delta.
\end{equation}

\begin{lemma}[Counting monomials]\label{lem:counting}
Let $U\subseteq\N^m$ be upward closed with $0\notin U$,
and set $V=\N^m\setminus U$ and $d=d_V=\max\{|J|:\N^J\subseteq V\}$.
Then 
 \begin{equation}\label{eq:counting-bounds}
 \binom{q+d}{d}
 \leq\#\{\alpha=(\alpha_1,\dots,\alpha_m)\in V:|\alpha|\leq q\}
 \leq C(q+1)^d\qquad(q\in\N)
\end{equation}
for a positive integer $C$ determined by the minimal elements of
$U$.
\end{lemma}

\begin{proof}
Let $J$ be a subset of $\{1,\dots,m\}$ such that $|J|=d$ and
$\mathbb N^J\subseteq V$. Then
\[
\#\{\alpha=(\alpha_1,\dots,\alpha_m)\in V:|\alpha|\leq q\}
\geq
\#\left\{
\alpha\in\mathbb N^J:|\alpha|\leq q
\right\}
=
\binom{q+d}{d}.
\]

Dickson's lemma implies that $U$ has finitely
many  minimal elements
\[
\gamma^{(1)},\ldots,\gamma^{(r)},\quad \text{where }\ \gamma^{(i)}=( \gamma^{(i)}_1,\dots,\gamma^{(i)}_m),
\] with respect to the componentwise monomial order.
\[
 M=1+\max\bigl\{\gamma_i^{(\ell)}:1\leq\ell\leq r,\ 1\leq i\leq m\bigr\}.
\]
For $\alpha=(\alpha_1,\dots,\alpha_m)\in V$, set
$J_\alpha
:=
\{i:\alpha_i\geq M\}.$
We claim that
\[
|J_\alpha|\leq d.
\]

Suppose, to the contrary, that $|J_\alpha|>d$. By the definition of $d$,
we have $\mathbb N^{J_\alpha}\not\subseteq V,$
and hence there exists $\beta=(\beta_1,\dots,\beta_m)\in\mathbb N^{J_\alpha}\cap U.$
Since the elements $\gamma^{(1)},\ldots,\gamma^{(r)}$ are the minimal
elements of $U$, there exists $\ell$ such that
\[
\gamma^{(\ell)}\leq\beta
\]
componentwise. Thus $\gamma^{(\ell)}$ is supported in $J_\alpha$.
Moreover, by the choice of $M$,
\[
\gamma_i^{(\ell)}<M\leq\alpha_i
\]
for every $i\in J_\alpha$.   Hence $\gamma^{(\ell)}\leq\alpha.$
Since $U$ is upward closed, this implies $\alpha\in U,$
contrary to $\alpha\in V$. Thus indeed
\[
|J_\alpha|\leq d.
\]
\begin{enumerate}
    \item[(a)] \textbf{Split $V$ according to the large coordinates.}

    For each $\alpha\in V$ with $|\alpha|\le q$, if $i\notin J_\alpha$, then from the definition $J_\alpha:=\{i:\alpha_i\ge M\}$ we have
    \[
    0\le \alpha_i<M.
    \]
    Hence the coordinate $\alpha_i$ has only $M$ possible values: $ \alpha_i\in\{0,1,\ldots,M-1\}.$ 
    On the other hand, if $i\in J_\alpha$, then
    \[
    M\le \alpha_i\le q.
    \]
    In particular, $\alpha_i$ has at most $q-M+1$ possible values, and hence
    at most $q+1$ possible values.

    Thus, for a fixed subset
    \[
    J\subseteq\{1,\ldots,m\},
    \qquad |J|=r\le d,
    \]
    the number of possible exponent vectors $\alpha$ satisfying
    \[
    J_\alpha=J,\qquad |\alpha|\le q
    \]
    is at most
    \[
    M^{m-r}(q+1)^r.
    \]
    \item[(b)] \textbf{Sum over all possible subsets $J$.}

    For each $e\le d$, there are
    \[
    \binom{m}{e}
    \]
    subsets $J_e\subseteq\{1,\ldots,m\}$ of cardinality $e$.  For each $\alpha\in V$, $J_\alpha=J_e$ for some $J_e$.    Therefore
    \[
    \begin{aligned}
    \#\{\alpha\in V:|\alpha|\le q\}
    &\le
    \sum_{e=0}^{d}
    \binom{m}{e}M^{m-e}(q+1)^e\\
    &\le \sum_{e=0}^{d}
    \binom{m}{e}M^{m-e}(q+1)^d\\
    &=C(q+1)^d\quad \text{where} \ \ C=\sum_{e=0}^{d}
    \binom{m}{e}M^{m-e}.
    \end{aligned}
    \]
    \end{enumerate}
This proves the lemma.
\end{proof}

\begin{definition}[Asymptotic notation]
   For two nonnegative  functions $f(n)$ and $g(n)$, we write:
  \begin{enumerate}
    \item[(i)] \(f(n)=\Omega(g(n))\) if there exist constants
    \(C>0\) and \(n_0\in\mathbb{N}_0\) such that
    \[
    f(n)\geq Cg(n)
    \qquad\text{for all } n\geq n_0.
    \]

    \item[(ii)] \(f(n)=O(g(n))\) if there exist constants
    \(D>0\) and \(n_0\in\mathbb{N}_0\) such that
    \[
    f(n)\leq Dg(n)
    \qquad\text{for all } n\geq n_0.
    \]

    \item[(iii)] \(f(n)=\Theta(g(n))\) if there exist constants
    \(C,D>0\) and \(n_0\in\mathbb{N}_0\) such that
    \[
    Cg(n)\leq f(n)\leq Dg(n)
    \qquad\text{for all } n\geq n_0.
    \]
\end{enumerate}
\end{definition}

\begin{theorem}[Growth and independence]\label{thm:growth}
Let $S/D$ be automorphically finitely generated. With the preceding
notation, we have
\begin{equation}\label{eq:main-equality}
 \atrdeg_D(S)=\dL(S)=d_\Delta.
\end{equation}
This is a finite nonnegative integer, and a subtuple of the coordinate images
$\bx_1,\ldots,\bx_m$ attains it. If its value is $d$, then
\[
 h_S(q)=\Theta(q^{ \atrdeg_D(S)}),\qquad
 \lim_{q\to\infty}\frac{\log h_S(q)}{\log q}=d.
\]
\end{theorem}

\begin{proof}
If $\N^J\subseteq\Delta$, Lemma~\ref{lem:standard} shows that
$(\bx_j)_{j\in J}$ is automorphically independent. Therefore
\[
 d_\Delta\leq\atrdeg_D(S)\leq\dL(S).
\]
Again, by Lemma~\ref{lem:standard} $$h_S(q)=\# \big\{\alpha=(\alpha_1,\dots,\alpha_m)\in \Delta:|\alpha|\leq q\big\}.$$    
Denote $d=d_\Delta$, it follows from Lemma \ref{lem:counting} that   
    \[
     \binom{q+d}{d}
 \leq h_S(q)
 \leq C(q+1)^{d}.
    \]
    Since
    \[
    \binom{q+d}{d}
    =
    \frac{(q+1)(q+2)\cdots(q+d)}{d!}=\frac{q^d}{d!}(1+\frac{1}{q})(1+\frac{2}{q})\cdots(1+\frac{d}{q})
    \sim
    \frac{q^d}{d!}\quad (q\rightarrow \infty),
    \]   
    and 
    $$C(q+1)^d\sim Cq^d\quad (q\rightarrow \infty),$$
    we obtain
    \[
    h_S(q)=\Theta(q^d).
    \]
Next, we prove $$d_\Delta\ge \mathcal{L}_D(S).$$
Let $t_1,\ldots,t_r\in S$ be any tuple that is left algebraically
independent over $D$. Choose $c\geq1$ such that $t_i\in F_cS
\ 
(i=1,\ldots,r).$
Then the monomials $t_1^{i_1}\cdots t_r^{i_r}\
(i_1+\cdots+i_r\leq q) $
are left $D$-linearly independent, and all belong to $F_{cq}S$ due to inclusion (\ref{inclusion1}).
Therefore
\[
h_S(cq)
\geq
\binom{q+r}{r}\sim\frac{q^r}{r!}.
\]
It follows that $r\leq d.$ 
By taking the maximal possible $r$, we obtain $\dL(S)\leq d.$ Hence, we obtain that
\[
 d_\Delta=\atrdeg_D(S)=\dL(S)
\]
and
\[
h_S(q)=\Theta(q^{\operatorname{atrdeg}^{\,l}_D(S)}).
\]
\end{proof}

The equality with $\dL(S)$ explains the numerical role of the
automorphic requirement. In the present class it does not change the
maximum. Its structural role is to identify the generated subring
with a skew polynomial ring, as in Lemma~\ref{lem:subring}.

\begin{proposition}[Hilbert polynomial and series]\label{prop:hilbert}
Let $\gamma^{(1)},\ldots,\gamma^{(r)}$ be the minimal elements of
$\Lambda$. The case $r=0$ is allowed. For $L\subseteq\{1,\ldots,r\}$,
let $\gamma_L$ be the componentwise maximum of the $\gamma^{(\ell)}$
with $\ell\in L$, and put $\gamma_\varnothing=0$.
Define
\[
 B_m(t)=
 \begin{cases}\binom{t+m}{m},&t\geq0,\\0,&t<0.\end{cases}
\]
Then, for every $q\in\N$,
\begin{equation}\label{eq:hilbert-inclusion-exclusion}
 h_S(q)=\sum_{L\subseteq\{1,\ldots,r\}}
              (-1)^{|L|}B_m(q-|\gamma_L|),
\end{equation}
and, as an identity of formal power series, 
\begin{equation}\label{eq:hilbert-series}
 \sum_{q\geq0}h_S(q)z^q
 =\frac{\displaystyle\sum_{L\subseteq\{1,\ldots,r\}}
                    (-1)^{|L|}z^{|\gamma_L|}}
        {(1-z)^{m+1}}.
\end{equation}
In particular, $h_S(q)$ agrees for all sufficiently large $q$ with a
polynomial $H_S(q)\in\Q[q]$ of degree $\atrdeg_D(S)$ and positive
leading coefficient.
\end{proposition}

{
\begin{proof}

{\bf 1) The standard stars-and-bars argument.}
Since $\Lambda \subseteq \N^m$ is upward closed and
$\gamma^{(1)},\ldots,\gamma^{(r)}$ are its minimal elements, we have
\[
\Lambda
=
\bigcup_{\ell=1}^r
\left(
\gamma^{(\ell)}+\N^m
\right).
\]
For a nonempty subset
\[
L \subseteq \{1,\ldots,r\},
\]
let $\gamma_L$ denote the coordinatewise maximum of the vectors
$\gamma^{(\ell)}$, $\ell \in L$. Then
\[
\bigcap_{\ell \in L}
\left(
\gamma^{(\ell)}+\N^m
\right)
=
\gamma_L+\N^m.
\]
Indeed, an exponent $\alpha=(\alpha_1,\ldots,\alpha_m)$ belongs to the
intersection if and only if
\[
\alpha_i \ge \gamma_i^{(\ell)}
\qquad
\text{for every } i=1,\ldots,m
\text{ and every } \ell\in L,
\]
which is equivalent to
\[
\alpha_i
\ge
\max_{\ell\in L}\gamma_i^{(\ell)}
=
(\gamma_L)_i
\qquad
\text{for every } i.
\]

For each $q\ge 1$, denote
\[U_{q,\gamma_L}=\big\{
\alpha \in \gamma_L+\N^m:\ |\alpha|\le q\big\}.
\]
Every $\alpha\in U_{q,\gamma_L}$ can be written uniquely as
\[
\alpha=\gamma_L+\beta,
\qquad
\beta\in\N^m.
\]
Since
\[
|\alpha|
=
|\gamma_L|+|\beta|,
\]
the condition $|\alpha|\le q$ is equivalent to
\[
|\beta|
\le
q-|\gamma_L|.
\]
By the standard stars-and-bars argument, we obtain
\begin{equation}\label{eq: bars_stars}
    \#U_{q,\gamma_L}=B_m(q-|\gamma_L|).
\end{equation}

{
{\bf 2) Inclusion-exclusion principle.} Let $T_q=\{\alpha\in\N^m:\ |\alpha|\le q\}$
be the set of all monomial exponents of total degree at most \(q\) correspond to \(T_q\). 
We will count 
\[
T_q\cap\Delta=T_q\setminus\Lambda.
\]

Since
\[
\Lambda=\bigcup_{\ell=1}^r A_\ell,
\qquad
A_\ell=\gamma^{(\ell)}+\N^m,
\]
we need to count the set
\[
T_q\setminus\Lambda=T_q\setminus\bigcup_{\ell=1}^r A_\ell=T_q\setminus \bigcup_{\ell=1}^r (T_q\cap A_\ell).
\]
The inclusion--exclusion principle gives
$$ \#\bigg(\bigcup_{\ell=1}^r (T_q\cap A_\ell)\bigg)=\sum_{\emptyset\ne L\subseteq\{1,\ldots,r\}}
(-1)^{|L|}
\#
\left(
T_q\cap\bigcap_{\ell\in L}A_\ell
\right).$$ Then, 
\[
\#\left(T_q\setminus\bigcup_{\ell=1}^r A_\ell\right)
=
\sum_{ L\subseteq\{1,\ldots,r\}}
(-1)^{|L|+1}
\#
\left(
T_q\cap\bigcap_{\ell\in L}A_\ell
\right).
\]
For each \(L\neq\varnothing\),
\[
\bigcap_{\ell\in L}A_\ell
=
\gamma_L+\N^m.
\]
Combining this with (\ref{eq: bars_stars}), we obtain that 
\[
\#
\left(
T_q\cap\bigcap_{\ell\in L}A_\ell
\right)=\#U_{q,\gamma_L}
=
B_m(q-|\gamma_L|).
\]
Therefore \begin{equation}\label{eq:inclusion_exclusion}
    \#(T_q\cap\Delta)=\#(T_q\setminus\Lambda)=\#\left(T_q\setminus\bigcup_{\ell=1}^r A_\ell\right)=\sum_{L\subseteq\{1,\dots,r\}}(-1)^{|L|+1}B_m(q-|\gamma_L|).
\end{equation}
Recall that Lemma~\ref{lem:standard} states that the residue classes
\[
\left\{\overline{x^\alpha} : \alpha \in \Delta\right\}
\]
form a left $D$-basis of $S$ and
$h_S(q)
=\#\left\{
\alpha \in \Delta : |\alpha| \le q
\right\}$. Hence, we have that
\begin{align*}
    h_S(q)&=\#\left\{
\alpha \in \Delta : |\alpha| \le q
\right\}\\&=\#(T_q\cap \Delta)\\
&=
\sum_{L\subseteq\{1,\ldots,r\}}
(-1)^{|L|+1}
B_m(q-|\gamma_L|).
\end{align*}}

{\bf 3) Hilbert series.}
We next compute the generating series. For every integer $a\ge 0$,
\[
\sum_{q\ge 0} B_m(q-a)z^q
=
\sum_{q\ge a}
\binom{q-a+m}{m}z^q.
\]
Setting $n=q-a$, we get
\[
\sum_{q\ge 0} B_m(q-a)z^q
=
z^a
\sum_{n\ge 0}
\binom{n+m}{m}z^n.
\]
Using the standard identity
\[
\sum_{n\ge 0}
\binom{n+m}{m}z^n
=
\frac{1}{(1-z)^{m+1}},
\]
we obtain
\[
\sum_{q\ge 0} B_m(q-a)z^q
=
\frac{z^a}{(1-z)^{m+1}}.
\]

Applying this identity with $a=|\gamma_L|$ and summing over all
subsets $L\subseteq\{1,\ldots,r\}$ gives
\[
\sum_{q\ge 0} h_S(q)z^q
=
\frac{
\displaystyle
\sum_{L\subseteq\{1,\ldots,r\}}
(-1)^{|L|+1}
z^{|\gamma_L|}
}{
(1-z)^{m+1}
}.
\]

Finally, let
\[
Q
=
\max_{L\subseteq\{1,\ldots,r\}}
|\gamma_L|.
\]
For every $q\ge Q$, each term
\[
B_m(q-|\gamma_L|)
=
\binom{q-|\gamma_L|+m}{m}
\]
is a polynomial in $q$. Since only finitely many subsets $L$ occur,
the function $h_S(q)$ agrees, for all sufficiently large $q$, with a
polynomial $H_S(q)\in\mathbb{Q}[q]$.

Let
\[
e=\deg H_S.
\]
Since the leading coefficient of $H_S$ is nonzero and
$H_S(q)=h_S(q)>0$ for all  $q\ge Q$, this leading
coefficient is positive. Consequently,
\[
h_S(q)=\Theta(q^e).
\]
On the other hand, Theorem~\ref{thm:growth} gives
\[
h_S(q)
=
\Theta\!\left(
q^{\operatorname{atrdeg}_D(S)}
\right).
\]
Therefore the two growth exponents must coincide, and hence
\[
e
=
\operatorname{atrdeg}_D(S).
\]
Thus the eventual Hilbert polynomial of $S$ has degree
$\operatorname{atrdeg}_D(S)$ and positive leading coefficient.
\end{proof}}

\begin{proposition}[Presentation independence]\label{prop:filtrations}
Two automorphic presentations of $S/D$ give linearly equivalent
degree filtrations: if they are $F$ and $G$, there exist positive
integers $a,b$ such that
\[
 F_qS\subseteq G_{aq}S,\qquad G_qS\subseteq F_{bq}S
 \qquad(q\in\N).
\]
Consequently the relative growth dimension
\[
 \rgdim_D(S):=\limsup_{q\to\infty}
                 \frac{\log\ldim{F_qS}}{\log q}
\]
is well defined and equals $\atrdeg_D(S)$. The limsup is a limit,
and either left or right dimensions give the same value.
\end{proposition}

\begin{proof}
All coordinate generators of the first presentation belong to
$G_aS$ for some $a\geq1$. Since $D=G_0S$ and $G$ is multiplicative,
every monomial of $F$-degree at most $q$ belongs to $G_{aq}S$.
Interchanging the presentations proves the second inclusion.
The remaining assertions follow from Lemma~\ref{lem:standard} and
Theorem~\ref{thm:growth}.
\end{proof}

{
\begin{corollary}\label{cor:elementary}
Let $S/D$ be automorphically finitely generated.
\begin{enumerate}
\item $\atrdeg_D(S)=0$ if and only if $S$ is finite-dimensional as a
left $D$-vector space; equivalently, it is finite-dimensional on the
right.
\item If $J$ is a proper two-sided ideal of $S$, then
$\atrdeg_D(S/J)\leq\atrdeg_D(S)$, where $D$ is identified with its
image in $S/J$.
\item Let  $z_1,\ldots,z_r$ be central indeterminates over $D$. Then
\[
 \atrdeg_D(S[z_1,\ldots,z_r])=\atrdeg_D(S)+r.
\]
\end{enumerate}
\end{corollary}}

\begin{proof}
For (i), the filtered dimensions are bounded exactly when their
integer growth degree is zero. An increasing exhaustive sequence of
vector spaces of bounded dimension stabilizes. Equality of the  left and right dimensions follows from the standard-monomial bases.
For (ii), the quotient filtration has no larger filtered dimensions,
and the quotient is again automorphically finitely generated.
A proper ideal has zero intersection with the division subring $D$.
For (iii), adjoining $r$ central variables gives a presentation with
coefficient automorphisms $\sigma_1,\ldots,\sigma_m,\id_D,\ldots,\id_D$.
Its standard exponents can be taken to be $\Delta\times\N^r$.
The largest coordinate subspace in this set has dimension
$d_\Delta+r$, so Theorem~\ref{thm:growth} applies.
\end{proof}

\section{Algebraicity, integral extensions, and Gelfand-Kirillov dimension}
\label{sec:extensions}

\subsection{Algebraicity over a maximum independent family}

\begin{definition}\label{def:algebraic-integral}
Let $B\subseteq S$. An element $u\in S$ is \emph{left algebraic over
$B$} if there are coefficients $b_0,\ldots,b_r\in B$, not all zero,
such that
\[
 b_0+b_1u+\cdots+b_ru^r=0.
\]
It is \emph{left integral over $B$} if it satisfies such a relation
with $r\geq1$ and $b_r=1$. The extension is left algebraic, or left
integral, if every element has the respective property. These
expressions are evaluated with coefficients on the indicated side;
no commutation between $u$ and $B$ is imposed.
\end{definition}

Over a division ring a nonzero leading coefficient can be inverted
on the left. Over a polynomial subring it need not be invertible.
The distinction in Definition~\ref{def:algebraic-integral} will be
used throughout this section.

\begin{theorem}[Effective left algebraicity]\label{thm:algebraicity}
Let $S/D$ be automorphically finitely generated with
$\atrdeg_D(S)=n$. Let $s_1,\ldots,s_n$ be automorphically independent
and put $B=D[s_1,\ldots,s_n]$. Then
\begin{itemize}
    \item[(i)] Fix a degree filtration $F$ on $S$.
For $u\in S$, choose $c\geq1$ with
$s_1,\ldots,s_n,u\in F_cS$. If $q\geq1$ satisfies
\begin{equation}\label{eq:exact-degree-bound}
 \binom{q+n+1}{n+1}>h_S(cq),
\end{equation}
then there is a nonzero relation
\begin{equation}\label{eq:effective-relation}
 \sum_{j=0}^q b_j u^j=0,\qquad b_j\in B,
 \qquad \deg_s b_j\leq q-j\ \text{when }b_j\ne0,
\end{equation}
with at least one $b_j\ne0$ for $j\geq1$. Here $\deg_s$ is total
degree in the skew polynomial presentation of $B$.
\item[(ii)] $S$ is left algebraic over $B$.
\end{itemize}
\end{theorem}

{
\begin{proof}
Consider the indexed family
\[
 \bigl(s^\alpha u^j\bigr)_{
       (\alpha,j)\in\N^n\times\N,\ |\alpha|+j\leq q}.
\]
It has $\binom{q+n+1}{n+1}$ members, all in $F_{cq}S$.
Under \eqref{eq:exact-degree-bound}, there is a nontrivial left
$D$-linear dependence
\[
 \sum_{j=0}^q\ \sum_{|\alpha|\leq q-j}
      d_{\alpha,j}s^\alpha u^j=0.
\]
Set $b_j=\sum_{|\alpha|\leq q-j}d_{\alpha,j}s^\alpha$.
The independence of the $s$-monomials ensures that not all $b_j$
vanish. If all $b_j$ with $j\geq1$ vanished, the relation would give
$b_0=0$ as well, a contradiction. The degree bounds are immediate
from Lemma~\ref{lem:subring}.

By Theorem~\ref{thm:growth}, there is a fixed constant $C>0$ such that $h_S(q)<Cq^n$ for all sufficiently large $q$. Then $h_S(cq)<(Cc^n)q^n$ for all sufficiently large $q$. It follows that $$ \binom{q+n+1}{n+1}\geq\frac{1}{(n+1)!}q^{n+1}>(Cc^n)q^n>h_S(cq)$$ for all sufficiently large $q$. Hence, for any $u\in S$,  inequality
\eqref{eq:exact-degree-bound} holds for all sufficiently large $q$.
\end{proof}}

\begin{remark}\label{rem:algebraicity-right}
No commutation of $u$ with the $s_i$ was used. Applying the same
argument to the family $(u^j s^\alpha)$, using right independence and
right dimensions, shows that $S$ is also right algebraic over $B$.
\end{remark}

\subsection{Finite extensions and elementwise integrality}

\begin{theorem}[Finite-extension invariance]\label{thm:finite}
Let $D\subseteq R\subseteq S$, where $R/D$ is automorphically
finitely generated. If $S$ is finite as a left $R$-module, then
\begin{equation}\label{eq:finite-invariance}
 \atrdeg_D(S)=\dL(S)=\atrdeg_D(R).
\end{equation}
No automorphic finite-generation assumption is imposed on $S$.
\end{theorem}

\begin{proof}
Put $n=\atrdeg_D(R)$ and choose a degree filtration on $R$ with
$h_R(q)\leq C(q+1)^n$. Monotonicity gives
$n\leq\atrdeg_D(S)\leq\dL(S)$.
Let $b_1,\ldots,b_m\in S$ be any commuting tuple that  is left
algebraically independent over $D$. Choose left $R$-module generators
$v_1,\ldots,v_t$ for $S$ with $v_1=1$. Express
\[
 v_i b_j=\sum_{\ell=1}^t r_{ij\ell}v_\ell,
 \qquad r_{ij\ell}\in R.
\]
Choose $c\geq1$ such that all $r_{ij\ell}$ lie in $F_cR$, and put
$G_q=\sum_{i=1}^t(F_qR)v_i$. Associativity gives
\[
 G_qb_j\subseteq G_{q+c}.
\]
Notice that this uses right multiplication on a left module and
requires no commutation between $b_j$ and $R$.
Starting at $1\in G_0$, every $b$-monomial of total degree at most
$q$ belongs to $G_{cq}$. Hence
\[
 \binom{q+m}{m}\leq\ldim{G_{cq}}
 \leq t\,h_R(cq)\leq tC(cq+1)^n.
\]
It follows that $m\leq n$, proving \eqref{eq:finite-invariance}.
\end{proof}

\begin{theorem}[Number of normalization variables]\label{cor:normalization-number}
If an extension $S/D$ admits a skew normalization subring
$B=D[s_1,\ldots,s_n]$, then
\[
 \atrdeg_D(S)=\dL(S)=n.
\]
In particular, any two skew normalization subrings of $S/D$
have the same number of variables.
\end{theorem}

\begin{proof}
By Lemma~\ref{lem:subring}, $B$ is an automorphically finitely
generated extension with $\atrdeg_D(B)=n$. Apply
Theorem~\ref{thm:finite} to $D\subseteq B\subseteq S$.
\end{proof}

\begin{corollary}[Diagonal matrix stabilization]\label{cor:matrix}
If $R/D$ is automorphically finitely generated and $r\geq1$, then
\[
 \atrdeg_{D_{\mathrm{diag}}}(M_r(R))=\atrdeg_D(R),
\]
where $D_{\mathrm{diag}}=\{dI_r:d\in D\}$.
\end{corollary}

{\begin{proof}
The scalar-matrix copy $R_{\mathrm{diag}}$ of $R$ is
an automorphically finitely generated extension of
$D_{\mathrm{diag}}$. The ring $M_r(R)$ is a left finite 
$R_{\mathrm{diag}}$-module. Apply Theorem~\ref{thm:finite}.
\end{proof}}
\begin{theorem}[Integral-extension invariance]\label{thm:integral}
Let $D\subseteq R\subseteq S$, with $R/D$ automorphically finitely
generated. If every element of $S$ is left integral over $R$, then
\[
 \atrdeg_D(S)=\dL(S)=\atrdeg_D(R).
\]
In particular, if $a_1,\ldots,a_n$ are automorphically independent
and $S$ is left integral over $D[a_1,\ldots,a_n]$, then
$\atrdeg_D(S)=n$.
\end{theorem}

\begin{proof}
Set $n=\atrdeg_D(R)$ and fix $h_R(q)\leq C(q+1)^n$ as above.
It suffices to bound the size $m$ of an arbitrary commuting tuple
$b_1,\ldots,b_m$ with left $D$-linearly independent monomials.
For each $j$, choose a monic relation
\begin{equation}\label{eq:monic-reductions}
 b_j^{d_j}+\sum_{k=0}^{d_j-1}c_{j,k}b_j^k=0,
 \qquad d_j\geq1,\quad c_{j,k}\in R.
\end{equation}
Choose $a\geq1$ with all $c_{j,k}\in F_aR$.
We claim that every monomial $b^\beta$ with $|\beta|\leq q$ belongs
to
\begin{equation}\label{eq:integral-bounded-space}
 \sum_{0\leq\gamma_j<d_j}(F_{aq}R)b^\gamma.
\end{equation}
If some $\beta_j\geq d_j$, commutativity among the $b_i$ permits us
to write $b^\beta=b_j^{d_j}v$. Substitution of
\eqref{eq:monic-reductions} lowers the total $b$-degree by at least
one, leaving every new coefficient on the left. Repeat inside each
remaining $b$-monomial. Along any branch there are at most $q$
reductions; the coefficients are products of at most $q$ elements of
$F_aR$ and hence lie in $F_{aq}R$. This proves
\eqref{eq:integral-bounded-space}. No coefficient is commuted past a
$b_i$ in this argument.

There are $d_1\cdots d_m$ bounded monomials in
\eqref{eq:integral-bounded-space}. Consequently
\[
 \binom{q+m}{m}\leq (d_1\cdots d_m)\,h_R(aq)
 \leq(d_1\cdots d_m)C(aq+1)^n.
\]
Thus $m\leq n$. Monotonicity and Theorem~\ref{thm:growth} for $R$
complete the proof.
\end{proof}

\begin{remark}\label{rem:finite-implies-integral}
Since $R$ is left noetherian, the left-finite hypothesis in
Theorem~\ref{thm:finite} implies the integral hypothesis in
Theorem~\ref{thm:integral}. Indeed, ${}_RS$ is noetherian, so for
$u\in S$ the ascending chain
\[
 R\subseteq R+Ru\subseteq R+Ru+Ru^2\subseteq\cdots
\]
stabilizes and gives a monic left relation for $u$.
The direct proof of Theorem~\ref{thm:finite} is retained because it
gives a transparent uniform growth bound from module generators
without invoking a noncommutative characteristic polynomial.
\end{remark}
\subsection{Gelfand-Kirillov dimension under a finiteness condition}

Lezama and Latorre define a generalized Gelfand--Kirillov dimension
for finitely semi-graded rings over a left noetherian degree-zero
domain, using Goldie dimension
\cite[Definition~4.1 and Proposition~4.3]{Lezama2017}. Over a division
ring, Goldie dimension agrees with vector-space dimension whenever
it is finite. Lezama and Venegas treat a related rank-based invariant
for algebras over commutative domains \cite{Lezama2020}. For an
arbitrary, possibly noncentral division subring, we use the following
explicit relative convention and prove the needed comparison directly.

\begin{definition}[Relative frame-based dimension]\label{def:relative-gk}
Let $D\subseteq B$ be a ring extension. A \emph{left $D$-frame} of
$B$ is a finite-dimensional left $D$-subspace $V\subseteq B$
containing $1$. For $q\geq1$, set
\[
 V^q=\Span_D\{v_1\cdots v_q:v_i\in V\}.
\]
The \emph{left Gelfand--Kirillov dimension} is
\begin{equation}\label{eq:relative-gk-definition}
 \GKdim_D(B)=\sup_V\limsup_{q\to\infty}
                 \frac{\log\ldim{V^q}}{\log q}.
\end{equation}
Here any infinite vector-space dimension is assigned the value
$+\infty$ in the growth expression. A frame is \emph{generating} if
$B$ is the subring generated by $D$ and $V$.
\end{definition}
\begin{theorem}[Generating-frame comparison]\label{thm:generating}
If $V$ is a generating left $D$-frame of $B$, then
\[
 \GKdim_D(B)=\limsup_{q\to\infty}
                  \frac{\log\ldim{V^q}}{\log q}.
\]
\end{theorem}

\begin{proof}
We follow the proof of \cite[Proposition 4.3]{Lezama2017}. Since $1\in V$, one has $D\subseteq V$ and $V^q\subseteq V^{q+1}$.
For positive integers $p,q$,
\[
 V^pV^q\subseteq V^{p+q}.
\]
To see this, expand the left coefficients in a product of two
spanned words and absorb an intervening coefficient from $D$ into
the first factor of the second word, which is still in $V$.
Consequently, $\bigcup_{q\geq1}V^q$ is the subring generated by $V$,
and hence equals $B$.

If some $V^q$ is infinite-dimensional, both sides of the desired
equality are $+\infty$. Otherwise, for any left $D$-frame $W$, choose
$c\geq1$ such that a finite left $D$-basis of $W$ lies in $V^c$.
Then $W\subseteq V^c$ and $W^q\subseteq V^{cq}$. Hence
\[
 \limsup_{q\to\infty}\frac{\log\ldim{W^q}}{\log q}
 \leq\limsup_{q\to\infty}
       \frac{\log(cq)}{\log q}
       \frac{\log\ldim{V^{cq}}}{\log(cq)}
 \leq\limsup_{n\to\infty}\frac{\log\ldim{V^n}}{\log n}.
\]
Taking the supremum over $W$ proves one inequality; the reverse
inequality follows by using $V$ itself as a frame.
\end{proof}

\begin{remark}\label{rem:gk-lower-independent}
For any commuting tuple $s_1,\ldots,s_r$ with left $D$-linearly
independent indexed monomials, the space
\[
 V=D+Ds_1+\cdots+Ds_r
\]
is a frame and
\[
 \ldim{V^q}\geq\binom{q+r}{r},\qquad
 \lim_{q\to\infty}\frac{\log\binom{q+r}{r}}{\log q}=r.
\]
Taking the supremum over all such finite tuples, including when
$\dL(S)=\infty$, gives
\begin{equation}\label{ine:GK1}
 \atrdeg_D(S)\leq\dL(S)\leq\GKdim_D(S).
\end{equation}
\end{remark}

\begin{theorem}[Finite extensions and Gelfand--Kirillov dimension]
\label{thm:gk-finite}
Let $D\subseteq R\subseteq S$, where $R/D$ is automorphically
finitely generated and $S$ is finite as a left $R$-module. Assume, in addition, that $S$ is locally finite as a $D$-bimodule,
in the following left-dimensional sense:
\begin{equation}\label{eq:gk-local-finiteness}
 \ldim{DsD}<\infty\qquad\text{for every }s\in S.
\end{equation}
Here $DsD$ denotes the additive span of the elements $asb$, with
$a,b\in D$. Then every power of every left $D$-frame in $S$ is
finite-dimensional, and
\[
 \GKdim_D(S)=\GKdim_D(R)=\atrdeg_D(R)=\atrdeg_D(S).
\]
\end{theorem}

\begin{proof}
Put $n=\atrdeg_D(R)$ and fix an automorphic degree filtration $F$ on
$R$. By Theorem~\ref{thm:growth}, $n$ is finite and
\[
 h_R(q)=\ldim{F_qR}=\Theta((q+1)^n).
\]
In particular, there is a constant $C>0$ such that
\begin{equation}\label{eq:gk-base-growth-bound}
 h_R(q)\leq C(q+1)^n\qquad(q\geq0).
\end{equation}

Let $W=F_1R$. Multiplicativity of $F$ gives $W^q\subseteq F_qR$.
Conversely, every coordinate monomial of degree at most $q$ is a
product of at most $q$ elements of $W$, and the remaining factors
can be chosen to be $1$. Taking left $D$-spans gives
\[
 W^q=F_qR\qquad(q\geq1).
\]
Consequently,
\[
 \lim_{q\to\infty}\frac{\log\ldim{W^q}}{\log q}=n.
\]

Since $W$ is a frame of both $R$ and $S$, it follows that
\begin{equation}\label{eq:gk-lower-bound}
 n\leq\GKdim_D(R)\leq\GKdim_D(S).
\end{equation}
Choose left $R$-module generators $v_1,\ldots,v_m$ for $S$, with
$v_1=1$, and set
\[
 U=\sum_{i=1}^m Dv_iD.
\]
By \eqref{eq:gk-local-finiteness}, $U$ is finite-dimensional on the
left over $D$. Moreover,
\[
 1\in U,\qquad UD\subseteq U,\qquad RU=S.
\]
Choose a left $D$-basis $u_1,\ldots,u_t$ of $U$, and define
\[
 G_q=\sum_{i=1}^t(F_qR)u_i\qquad(q\geq0).
\]
Then
\begin{equation}\label{eq:gk-module-filtration}
 1\in G_0,\qquad G_qD\subseteq G_q,\qquad
 \ldim{G_q}\leq t\,h_R(q).
\end{equation}
The right $D$-stability follows from $UD\subseteq U$.

Let $V$ be an arbitrary left $D$-frame of $S$, and choose left
$D$-generators $a_1,\ldots,a_\ell$ for $V$.
For each $i,j$, choose coefficients in $R$ such that
\[
 u_i a_j=\sum_{k=1}^t r_{ijk}u_k,\quad r_{ijk}\in R.
\]
There are only finitely many coefficients $r_{ijk}$, so choose an
integer $c\geq1$ such that every $r_{ijk}\in F_cR$. Thus $V\subseteq G_c$.
Multiplicativity of the filtration on $R$ gives
\[
 G_q a_j\subseteq G_{q+c}\qquad(q\geq0).
\]
If $a=\sum_j d_j a_j\in V$, with $d_j\in D$, then
\[
 G_q a\subseteq\sum_j(G_qd_j)a_j
          \subseteq G_{q+c},
\]
where the second inclusion uses $G_qD\subseteq G_q$.
Thus $G_qV\subseteq G_{q+c}$ for every $q\geq0$.
Starting with $1\in G_0$, induction yields
\[
 V^q\subseteq G_{cq}\qquad(q\geq1).
\]
It follows from \eqref{eq:gk-base-growth-bound} and
\eqref{eq:gk-module-filtration} that
\[
 \ldim{V^q}\leq\ldim{G_{cq}}
       \leq t\,h_R(cq)
       \leq tC(cq+1)^n.
\]
In particular, $V^q$ is finite-dimensional, and
\[
 \limsup_{q\to\infty}
 \frac{\log\ldim{V^q}}{\log q}\leq n.
\]
Taking the supremum over all frames $V$ gives
$\GKdim_D(S)\leq n$. Together with \eqref{eq:gk-lower-bound} and (\ref{ine:GK1}), this
proves
\[
 \GKdim_D(S)=\GKdim_D(R)=\atrdeg_D(R)=\atrdeg_D(S).
\]
\end{proof}

\begin{remark}[The role of local finiteness]
\label{rem:gk-local-finiteness}
The additional hypothesis is necessary for the stated frame-based
invariant to be finite. Indeed, for any $s\in S$, the space
$V=D+Ds$ is a left $D$-frame and $DsD\subseteq V^2$.
Thus, if $DsD$ is infinite-dimensional for some $s$, then $V^q$ is
infinite-dimensional for every $q\geq2$. Therefore  $ \GKdim_D(S)=\infty$.
\end{remark}

\begin{corollary}
\label{cor:gk-ore-quotients}
If $S/D$ is automorphically finitely generated, then
\[
 \GKdim_D(S)=\atrdeg_D(S).
\]
\end{corollary}

\begin{proof}
Let $F$ be an automorphic degree filtration on $S$.
For every $s\in F_qS$, multiplicativity gives
\[
 DsD\subseteq(F_0S)(F_qS)(F_0S)\subseteq F_qS.
\]
Thus $S$ satisfies \eqref{eq:gk-local-finiteness}.
Apply Theorem~\ref{thm:gk-finite} with $R=S$.
Equivalently, use $(F_1S)^q=F_qS$ and
Theorem~\ref{thm:growth} directly.
\end{proof}


\subsection{Why the hypotheses cannot simply be discarded}

\begin{example}[Algebraicity does not preserve the degree]
\label{ex:algebraic-not-stable}
Consider the following extension of rings
$$k\subset A\subset S,$$
where $k$ is a field and $ S=k[x,y,z]/(xy,xz)$ and $A=k[x]\subseteq S$. Here we use the same letters for the residue classes. Both $A$ and $S$ are automorphically finitely generated over $k$. We will show that $S$ is algebraic over $A$ but $\atrdeg_k(S)>\atrdeg_k(A)$. Thus, in Theorem \ref{thm:integral}, the condition that $S$ is left integral over $R$ cannot be replaced by the condition that $S$ is left algebraic over $R$.

Indeed, every $s\in S$ has
a unique expression
\[
 s=f(x)+g(y,z),\qquad g(0,0)=0.
\]
Because $xg(y,z)=0$, it satisfies
\[
 xs-xf(x)=0.
\]
The coefficient  $x\in A$ of $s$ is nonzero, so $S$ is left algebraic over
$A$. Nevertheless, $\atrdeg_k(A)=1$, whereas $y,z$ are algebraically
independent over $k$ and
\[
 h_S(q)=\binom{q+2}{2}+q,\qquad\atrdeg_k(S)=2.
\]
This also shows that $\{x\}$ is an independent family maximal under
inclusion but not of maximum size: every pair $x,s$ satisfies the
nonzero polynomial relation displayed above.
\end{example}

The following example shows that the hypothesis that \(R/D\) is automorphically finitely generated cannot be omitted from Theorem~\ref{thm:finite}.
{
\begin{example}
\label{ex:free-counterexample}
Let $k$ be a field, $A=k\langle x,y\rangle$ the free associative
algebra, and $B=k[x',y']$. Let $\pi:A\to B$ send $x$ to $x'$ and $y$
to $y'$, and define
\[
 S=A\times B,\qquad
 R=\{(a,\pi(a)):a\in A\},\qquad
 D=\{(\lambda,\lambda):\lambda\in k\}\ \simeq k.
\]
The elements $(1,0)$ and $(0,1)$ generate $S$ as a left $R$-module.

By Bergman's centralizer theorem \cite{Bergman1969}, for any nonscalar element $u\in A$, its centralizer $C_A(u)\simeq k[t]$, the polynomial ring in the variable $t$. 

Suppose that $u=(u_0,\pi(u_0))$ and $v=(v_0,\pi(v_0))$, where $\ u_0,v_0\in A$, are commuting independent elements of $R$ over $k$. Then $\{u_0,v_0\}$ is also a pair of commuting independent elements of $A$ over $k$. It follows that  $ \{u_0,v_0\}\subseteq C_A(u_0)\simeq k[t]$. This is a contradiction. Therefore 
$\atrdeg_D(R)\le1$. The element $(x,x')\in R$ has independent powers and is automorphic over $D$ with respect to $(\operatorname{id}_D,0)$. So $\atrdeg_D(R)=1$.
 
On the other hand, $(0,x')$ and $(0,y')$ are automorphically independent
in $S$ 
Thus $\atrdeg_D(S)\geq2$. The ring $R$ is not
automorphically finitely generated over $D$: since $D$ is central in $R$,
such a presentation would make $R$ commutative. 
\end{example}}

\section{Outer-independent automorphisms and normalization}
\label{sec:outer}

For $c\in D^\times$, define $\Ad_c(d)=cdc^{-1}$ for all $d\in D$, which is called an inner automorphism of $D$. Write
$\Out(D)=\Aut(D)/\Inn(D)$ where $\Inn(D)=\{\Ad_c:c\in D^\times\}$. Throughout this section
\[
 P=D[x_1,\ldots,x_m;\sigma_1,\ldots,\sigma_m]
\]
is a skew polynomial ring in commuting indeterminates $x_1,\dots,x_m$ with zero derivations and $\sigma_1,\dots,\sigma_m\in\Aut(D)$. Thus
\begin{equation}\label{eq:skew-multiplication}
 (a x^\alpha)(b x^\beta)
 =a\sigma^\alpha(b)x^{\alpha+\beta}
 \qquad(a,b\in D,\ \alpha,\beta\in\N^m).
\end{equation}
The term ``automorphic element'' still has the derivation-inclusive
meaning of Definition~\ref{def:automorphic}

We impose the additional
hypothesis
\begin{equation}\label{eq:outer-independence}
 \Phi:\Z^m\longrightarrow\Out(D),\qquad
 \alpha\longmapsto[\overline{\sigma}^\alpha]:=[\sigma_1^{\alpha_1}\cdots\sigma_m^{\alpha_m}]
 \quad\text{is injective}.
\end{equation}
We call this \emph{outer independence}. Equivalently,
$[\sigma^\alpha]\ne[\sigma^\beta]$ whenever
$\alpha,\beta\in\N^m$ are distinct: every integer vector is a
difference of two nonnegative vectors.

\subsection{Separation of monomial components}

\begin{lemma}[Coefficient separation]\label{lem:separation}
Assume \eqref{eq:outer-independence}. If $W\subseteq P$ is an
additive subgroup stable under left and right multiplication by
$D$, then for every
\[
 f=\sum_{\alpha\in U}c_\alpha x^\alpha\in W,
 \qquad c_\alpha\in D^\times,
\]
each monomial $x^\alpha$, $\alpha\in U$, belongs to $W$.
Consequently $W$ is a direct sum of monomial components $Dx^\alpha$.
\end{lemma}

\begin{proof}
Induct on $|U|$. For $|U|=1$, multiply $f$ by the inverse of its
coefficient on the left. Suppose $|U|>1$. Fix an $\alpha=(\alpha_1,\dots,\alpha_m)\in U$. For every $d\in D$, note that $x^\alpha d=x_1^{\alpha_1}\dots x_m^{\alpha_m}d=\sigma_1^{\alpha_1}\cdots\sigma_m^{\alpha_m}(d)x^\alpha=\sigma^{\alpha}(d)x^\alpha$. Put $\tau=\Ad_{c_{\alpha}}\sigma^{\alpha}$. Then
\[
 g_d:=fd-\tau(d)f\in W,
\]
and the coefficient of $x^{\alpha}$ in $g_d$ is zero. For
$\beta\in U\setminus\{\alpha\}$ the coefficient of $x^\beta$ is
\[
 c_\beta\sigma^\beta(d)-\tau(d)c_\beta=\big((\Ad_{c_{\beta}}\sigma^\beta-\Ad_{c_{\alpha}}\sigma^\alpha)(d)\big)c_\beta.
\]
It cannot vanish for every $d$: that would give
$\tau=\Ad_{c_\beta}\sigma^\beta$ and hence
$[\sigma^{\alpha}]=[\sigma^\beta]$, contradicting
\eqref{eq:outer-independence}.
Choose a $d$ for which this coefficient is nonzero. The support of
$g_d$ has fewer than $|U|$ members, so induction gives $x^\beta\in W$.
Doing this for every $\beta\ne\alpha$ and subtracting the
corresponding terms from $f$ gives $x^{\alpha}\in W$ as well.
\end{proof}

\begin{proposition}[Monomial ideals and affine-monomial elements]
\label{prop:monomial-rigidity}
Under \eqref{eq:outer-independence}, the following hold.
\begin{enumerate}
\item Every two-sided ideal $I\subseteq P$ has the form
\[
 I=\bigoplus_{\alpha\in\Lambda}Dx^\alpha
\]
for an upward-closed subset $\Lambda\subseteq\N^m$. Conversely,
every such subset defines a two-sided ideal. For a proper ideal,
$0\notin\Lambda$, $I\cap D=\{0\}$, and this $\Lambda$ equals
$\LE(I)$ for every monomial order.
\item Let $S=P/I$ be a proper quotient and put $\Delta=\N^m\setminus\Lambda$.
An element $s\in S\setminus D$ is automorphic if and only if it has
a unique expression
\begin{equation}\label{eq:affine-monomial}
 s=b+c\overline{x^\alpha},\qquad
 b\in D,\quad c\in D^\times,\quad
 \alpha\in\Delta\setminus\{0\}.
\end{equation}
Its associated pair is
\begin{equation}\label{eq:affine-associated-pair}
 \tau=\Ad_c\sigma^\alpha,\qquad
 \varepsilon(d)=bd-\tau(d)b\quad(d\in D).
\end{equation}
Every element of $D$ is automorphic as well, although its associated
pair need not be unique.
\end{enumerate}
\end{proposition}

\begin{proof}
A two-sided ideal is stable under left and right multiplication by
$D$, so Lemma~\ref{lem:separation} proves its monomial decomposition.
Multiplication by the variables makes its support upward closed.
Conversely, \eqref{eq:skew-multiplication} shows that upward-closed
support is stable under multiplication on both sides by coefficients
and variables. The remaining assertions in part~(i) follow at once.

For part~(ii), write the standard expansion of a nonconstant
$s\in S$ as
\[
 s=b+\sum_{\alpha\in U}c_\alpha\overline{x^\alpha},
 \qquad\varnothing\ne U\subseteq\Delta\setminus\{0\},
 \quad c_\alpha\in D^\times.
\]
Suppose $sd=\tau(d)s+\varepsilon(d)$ for an associated pair
$(\tau,\varepsilon)$. Comparing the coefficients at each nonzero
exponent gives
\[
 c_\alpha\sigma^\alpha(d)=\tau(d)c_\alpha
 \qquad(d\in D,\ \alpha\in U),
\]
or \[
 \big(\Ad_{c_\alpha}\sigma^\alpha(d)-\tau(d)\big)c_\alpha=0
 \qquad(d\in D,\ \alpha\in U).
\]
Thus $[\tau]=[\sigma^\alpha]$ for all $\alpha\in U$. Outer
independence forces $U$ to have a single member, say $\alpha$.
The coefficient at exponent zero then gives
$\varepsilon(d)=bd-\tau(d)b$, and the nonzero coefficient gives
$\tau=\Ad_c\sigma^\alpha$. Uniqueness of the standard expansion
proves uniqueness in \eqref{eq:affine-monomial}.

Conversely, for any element $s$ of the form \eqref{eq:affine-monomial},
formula \eqref{eq:skew-multiplication} gives
\[
 sd=\tau(d)s+bd-\tau(d)b.
\]
The map $d\mapsto bd-\tau(d)b$ is a $\tau$-derivation, so $s$ is
automorphic with the pair in \eqref{eq:affine-associated-pair}.
The assertion for scalars was proved in
Remark~\ref{rem:uniqueness-warning}.
\end{proof}

In particular, under the stricter condition $sd=\tau(d)s$ with no
derivation term, every nonzero automorphic element is a scalar
multiple of a single monomial. For a nonconstant element in
\eqref{eq:affine-monomial}, a nonzero $b$ would make $\tau=\Ad_b$;
this would contradict outer independence because $\alpha\ne0$.
It is the derivation-inclusive definition that permits the constant
shift in \eqref{eq:affine-monomial}.

\subsection{A criterion and a description of all normalization subrings}

Let $S=P/I$ be proper and write $\bx_i=x_i+I$. Define
\begin{equation}\label{eq:surviving-axes}
 E=\{i:\bx_i\text{ is not nilpotent}\}
   =\{i:ke_i\in\Delta\text{ for every }k\in\N\},
 \qquad p=|E|,
\end{equation}
where $e_i$ is the $i$th standard basis vector of $\N^m$.
For $i\notin E$, choose $N_i\geq1$ such that $\bx_i^{N_i}=0$.
Every $J$ with $\N^J\subseteq\Delta$ is contained in $E$, so
Theorem~\ref{thm:growth} gives
\begin{equation}\label{eq:degree-less-axes}
 d:=\atrdeg_D(S)\leq p.
\end{equation}

\begin{theorem}[Normalization under outer independence]
\label{thm:outer-normalization}
Assume \eqref{eq:outer-independence}. For a proper quotient $S=P/I$,
the following are equivalent.
\begin{enumerate}
\item $S$ is automorphically normalizable over $D$.
\item $\atrdeg_D(S)=|E|$.
\item $\N^E\subseteq\Delta$.
\item The variables $(\bx_i)_{i\in E}$ are automorphically
independent and $S$ is finite as a left module over
$D[\bx_i:i\in E]$.
\end{enumerate}
When these conditions hold, every normalization tuple has, after
permutation, the form
\begin{equation}\label{eq:all-normalization-tuples}
 (b_i+c_i\bx_i^{a_i})_{i\in E},\qquad
 b_i\in D,\quad c_i\in D^\times,\quad a_i\geq1.
\end{equation}
Conversely, any commuting tuple of this form is a normalization
tuple. The normalization subrings are exactly
\begin{equation}\label{eq:all-normalization-subrings}
 B_{\boldsymbol a}=D[\bx_i^{a_i}:i\in E],\qquad a_i\geq1.
\end{equation}
For each of these subrings, $S$ has a left generating set of size at
most
\[
 \Bigl(\prod_{i\in E}a_i\Bigr)
 \Bigl(\prod_{i\notin E}N_i\Bigr).
\]
An empty product is interpreted as $1$.
\end{theorem}

\begin{proof}
We first prove (i)$\Rightarrow$(ii). Suppose that $S$ is finite as a
left module over a normalization subring
$B=D[s_1,\ldots,s_r]$. By
Proposition~\ref{prop:monomial-rigidity},
\[
 s_j=b_j+c_j\overline{x^{\alpha^{(j)}}},\qquad
 b_j\in D,\quad c_j\in D^\times,\quad
 \alpha^{(j)}\in\Delta\setminus\{0\}.
\]
No member of an independent tuple lies in $D$. Put
\[
 H=\N\alpha^{(1)}+\cdots+\N\alpha^{(r)}\subseteq\N^m.
\]
Since $b_j,c_j,c_j^{-1}\in D$,
\[
 B=D[\overline{x^{\alpha^{(1)}}},\ldots,
       \overline{x^{\alpha^{(r)}}}].
\]
The multiplication rule shows that every element of $B$ is
supported in $H\cap\Delta$.

Choose finitely many left $B$-module generators $v_1,\ldots,v_t$
for $S$, and let $U\subseteq\Delta$ be the union of their finite
standard supports. A product of standard monomials is either the
monomial whose exponent is sum of the exponents or zero. Coefficients remain
nonzero under the coefficient automorphisms. Consequently,
\begin{equation}\label{eq:support-cover}
 \Delta\subseteq U+H.
\end{equation}
Indeed, every standard monomial belongs to $S=\sum_\ell Bv_\ell$,
whereas every such sum has support contained in $U+H$.

Choose $k_0>\max\{|u|:u\in U\}$. For each $i\in E$, the exponent
$k_0e_i$ lies in $\Delta$, so
\[
 k_0e_i=u+\sum_{j=1}^r n_j\alpha^{(j)},
 \qquad u\in U,\quad n_j\in\N.
\]
All coordinates are nonnegative. Therefore every
$\alpha^{(j)}$ with $n_j>0$ is supported on the $i$th axis.
The choice of $k_0$ guarantees that some $n_j$ is positive. Hence
\begin{equation}\label{eq:pure2}
 \text{for each }i\in E\text{, some }j\text{ satisfies }
 \alpha^{(j)}=a e_i\quad\text{with }a\geq1.
\end{equation}
Different axes require different members of the tuple. It follows
that $p\leq r$. Independence gives $r\leq d$, and
\eqref{eq:degree-less-axes} gives $d\leq p$. Thus $p=r=d$.
In particular, only now can one conclude that \emph{every}
$\alpha^{(j)}$ is a pure power exponent on a distinct axis in $E$.
This proves both (ii) and the necessary tuple form
\eqref{eq:all-normalization-tuples}.

For (ii)$\Rightarrow$(iii), Theorem~\ref{thm:growth} provides a set
$J$ of size $d=p$ with $\N^J\subseteq\Delta$. Since $J\subseteq E$,
one has $J=E$.

If (iii) holds, the coordinate variables in $E$ are independent.
Every nonzero standard monomial has exponent less than $N_i$ in
each coordinate $i\notin E$, since $\bx_i^{N_i}=0$. As the
variables commute, $S$ is generated on the left over
$D[\bx_i:i\in E]$ by
\begin{equation}\label{eq:nilpotent-module-generators}
 \left\{\prod_{i\notin E}\bx_i^{b_i}:0\leq b_i<N_i\right\}.
\end{equation}
This proves (iii)$\Rightarrow$(iv); (iv)$\Rightarrow$(i) is the
definition.

For the subring classification, constants and invertible scalar
factors do not change the generated $D$-subring:
\[
 D[b_i+c_i\bx_i^{a_i}:i\in E]
 =D[\bx_i^{a_i}:i\in E]=B_{\boldsymbol a}.
\]
Conversely, under (iii), the tuple $(\bx_i^{a_i})_{i\in E}$ is
automorphically independent for every choice of positive integers
$a_i$. Division of exponents by $a_i$ on the axes in $E$ shows that
$S$ is generated over $B_{\boldsymbol a}$ by
\[
 \left\{
  \Bigl(\prod_{i\in E}\bx_i^{r_i}\Bigr)
  \Bigl(\prod_{i\notin E}\bx_i^{b_i}\Bigr):
  0\leq r_i<a_i,\quad 0\leq b_i<N_i
 \right\}.
\]
This proves the converse classification and the generator bound.

Finally, a commuting tuple as in
\eqref{eq:all-normalization-tuples} is automorphic by
Proposition~\ref{prop:monomial-rigidity}. Its ordered monomials have
distinct leading exponents $\sum_{i\in E}a_i\beta_i e_i$ and
nonzero leading coefficients. Those exponents all lie in
$\N^E\subseteq\Delta$. A largest-leading-exponent argument therefore
proves left independence. Since its generated subring is
$B_{\boldsymbol a}$, the tuple is a normalization tuple.
\end{proof}

\begin{corollary}[Domain quotients]\label{cor:domain-quotients}
Under \eqref{eq:outer-independence}, a proper quotient $P/I$ is a
domain if and only if $I$ is generated by a subset of the variables.
In that case, the quotient is a skew polynomial ring on the remaining
variables and is automorphically normalizable.
\end{corollary}

\begin{proof}
Suppose that $P/I$ is a domain and let $J=\{i:x_i\in I\}$.
If $x^\alpha\in I$, its image is a zero product of coordinate
variables. Since $I$ is proper, $\alpha\ne0$. Some variable in its
support must have zero image, so $x^\alpha\in(x_i:i\in J)$.
Proposition~\ref{prop:monomial-rigidity} gives
$I=(x_i:i\in J)$. Conversely, such a quotient is the skew polynomial
ring over $D$ on the complementary variables, and its leading-term
multiplication rule shows that it is a domain. Its coordinate tuple
is itself a normalization tuple.
\end{proof}

\begin{example}[A supply of outer-independent actions]
\label{ex:translations}
Let $D=\Q(t_1,\ldots,t_m)$ and define
\[
 \sigma_i(t_j)=t_j+\delta_{ij}.
\]
The automorphisms commute. Since $D$ is commutative,
$\Inn(D)=\{\id_D\}$, and $\sigma^\alpha$ sends $t_j$ to
$t_j+\alpha_j$ for $\alpha\in\Z^m$. Characteristic zero implies
\eqref{eq:outer-independence}. Thus the criterion and the complete
description of normalization subrings apply in every number of
variables, without a finiteness assumption on the automorphism
orders.
\end{example}

\section{Computation and examples}\label{sec:examples}

\subsection{The support formula and a computation}

\begin{proposition}[Computation from leading supports]
\label{prop:support-formula}
Let $S=P/I$ be an automorphic presentation and let
$\gamma^{(1)},\ldots,\gamma^{(r)}$ be the minimal elements of
$\LE(I)$. Put $E_\ell=\Supp(\gamma^{(\ell)})$ and $\mathcal{H}=\{E_1,\dots,E_r\}$ Then
\begin{equation}\label{eq:support-formula}
 \atrdeg_D(S)
  =\max\{|J|:J\subseteq\{1,\dots,m\},\ E_\ell\nsubseteq J\text{ for every }\ell\}
 =m-\tau(\mathcal H),
\end{equation}
where $\tau(\mathcal H)$ is the smallest size of a subset of $\{1,\ldots,m\}$ meeting every $E_\ell$. For an empty family, $\tau(\varnothing)=0$.
\end{proposition}

\begin{proof}
The inclusion $\N^J\subseteq\Delta$ holds if and only if no
$\gamma^{(\ell)}$ is supported in $J$. Thus the first equality is
Theorem~\ref{thm:growth}.  The condition that no $E_\ell$ be contained in $J$ is equivalent to the complement of $J$ meeting every $E_\ell$, proving the second equality.
\end{proof}

\begin{example}[A four-variable computation]\label{ex:four-variables}
Let
\[
 P=D[x_1,x_2,x_3,x_4;\sigma_1,\sigma_2,\sigma_3,\sigma_4],\qquad
 I=(x_1x_2,x_2x_3,x_4^3).
\]
These monomials generate a two-sided ideal for any commuting
$\sigma_i$. The leading supports are
\[
 \{1,2\},\qquad\{2,3\},\qquad\{4\}.
\]
A minimum set meeting all three is $\{2,4\}$, so
$\atrdeg_D(P/I)=4-\#\{2,4\}=2$, attained by $\bx_1,\bx_3$.
For the first three variables, the standard monomials are all
monomials in $\bar{x}_1,\bar{x}_3$, together with the positive powers of $\bar{x}_2$.
The exponent of $\bar{x}_4$ can be $0,1,$ or $2$ corresponding to the value of $j$ in the following sum. Consequently
\[
 h_{P/I}(q)=\sum_{j=0}^{\min\{2,q\}}
       \left(\binom{q-j+2}{2}+q-j\right),
\]
and for $q\geq2$,
\[
 H_{P/I}(q)=\frac{3q^2+9q-4}{2}.
\]
If the action is outer-independent, then $E=\{1,2,3\}$.
Its size is three, so this two-dimensional quotient is not
automorphically normalizable. In particular, computing the
relative dimension alone does not settle existence of a
normalization without further structural information.
\end{example}

\subsection{Dimension and normalization need not coincide as existence statements}

\begin{example}[Two quotients of dimension one]\label{ex:contrasting-quotients}
Take $m=2$ in Example~\ref{ex:translations}, and let
$P=D[x_1,x_2;\sigma_1,\sigma_2]$. The quotient
\[
 S_1=P/(x_1x_2)
\]
has standard basis $1,\bx_1,\bx_1^2,\ldots,
\bx_2,\bx_2^2,\ldots$. Thus
\[
 h_{S_1}(q)=2q+1,\qquad\atrdeg_D(S_1)=1,
\]
but both coordinate variables are nonnilpotent. It is not
automorphically normalizable. The obstruction for this quotient is
also covered by \cite[Lemma~5.10]{ParanVo2025}.

By contrast, for $N\geq1$ the quotient
\[
 S_2=P/(x_2^N)
\]
has $\atrdeg_D(S_2)=1$ and $E=\{1\}$. It is a free left module of
rank $N$ over $D[\bx_1]$, with basis
$1,\bx_2,\ldots,\bx_2^{N-1}$. Indeed, the corresponding products
with powers of $\bx_1$ are exactly its standard monomials.
The two rings have the same automorphic transcendence degree, but
only the second realizes that degree by a finite normalization.
\end{example}

\begin{example}[A skew Laurent obstruction]\label{ex:Laurent}
Let $D=\Q(t)$ and $\sigma(t)=t+1$. Consider the skew Laurent
polynomial ring
\[
 S=D[x^{\pm1};\sigma]
   =\bigoplus_{j\in\Z}Dx^j,\qquad
 (a x^i)(b x^j)=a\sigma^i(b)x^{i+j}
 \quad(a,b\in D,\ i,j\in\Z).
\]
In particular, $xd=\sigma(d)x$ and
$x^{-1}d=\sigma^{-1}(d)x^{-1}$ for $d\in D$. The presentation
\[
 S\cong D[u,v;\sigma,\sigma^{-1}]/(uv-1),\qquad
 \overline u\longmapsto x,\quad \overline v\longmapsto x^{-1},
\]
exhibits $S/D$ as a compatible Ore quotient with zero derivations.
The induced total-degree filtration is
\[
 F_qS=\bigoplus_{j=-q}^{q}Dx^j,
\]
so Theorem~\ref{thm:growth} gives
\[
 h_S(q)=2q+1\quad(q\in\N),\qquad \atrdeg_D(S)=1.
\]

We prove that $S$ is not skew normalizable over $D$, allowing
derivation terms in the normalization generators. Let
$f=\sum_{j\in\Z}c_jx^j\in S\setminus D$ be automorphic with an
associated pair $(\tau,\varepsilon)$, so that
\[
 fd=\tau(d)f+\varepsilon(d)\qquad(d\in D).
\]
For every $j\ne0$ with $c_j\ne0$, comparison of the $x^j$
coefficients yields
\[
 c_j\sigma^j(d)=\tau(d)c_j\qquad(d\in D).
\]
Since $D$ is commutative, this implies $\tau=\sigma^j$.
The powers of $\sigma$ are distinct, because $\sigma^j(t)=t+j$.
There is therefore exactly one nonzero exponent in the support of
$f$. Writing it as $a$, we obtain
\[
 f=b+cx^a,\qquad b\in D,\quad c\in D^\times,\quad
 a\in\Z\setminus\{0\}.
\]
The constant coefficient gives
\[
 \tau=\sigma^a,\qquad
 \varepsilon(d)=bd-\sigma^a(d)b\quad(d\in D).
\]
Conversely, every element of this form is automorphic with this
associated pair. In particular, the constant term $b$ cannot be
discarded under the present definition, but it does not change the
generated subring:
\[
 D[f]=D[x^a]=\bigoplus_{n\in\N}Dx^{an}.
\]

Suppose that $S$ admitted a skew normalization. By
Theorem~\ref{cor:normalization-number}, a normalization tuple would
have exactly one member $f$. This member lies outside $D$, so the
normalization subring would be $B=D[x^a]$ for some $a\ne0$.
Choose finitely many left $B$-module generators $v_1,\ldots,v_r$
of $S$, and let
\[
 A=\bigcup_{\ell=1}^{r}\Supp_x(v_\ell)\subseteq\Z,
 \qquad
 \Supp_x\left(\sum_jd_jx^j\right)=\{j:d_j\ne0\}.
\]
The set $A$ is finite. The Laurent multiplication rule shows that
every element of $\sum_{\ell=1}^r Bv_\ell$ has support contained in
\[
 A+a\N=\{u+an:u\in A,\ n\in\N\}.
\]
If $a>0$, this set is bounded below; if $a<0$, it is bounded above.
In either case it cannot contain all of $\Z$, whereas $x^j\in S$
for every $j\in\Z$. This contradicts finite generation over $B$.
Thus $S$ is not skew normalizable over $D$.

For the automorphism-only notion of normalization, the negative
assertion is a special case of \cite[Lemma~5.6]{ParanVo2025}.
The preceding argument establishes it for the derivation-inclusive
notion as well. Notice also that the action map
\[
 \Z^2\longrightarrow\Out(D),\qquad
 (r,s)\longmapsto[\sigma^{r-s}],
\]
has kernel $\Z(1,1)$. Thus the outer-independence hypothesis fails
for this presentation, and
Theorem~\ref{thm:outer-normalization} does not apply to it.
\end{example}.



\subsection{Limits of the automorphic condition}

\begin{example}
\label{ex:one-generator}
Let $F=\Q(t)$, let $\tau(t)=-t$, and set
\[
 S=F[x_1,x_2;\id_F,\tau],\qquad s=x_1+x_2.
\]
The powers of $s$ are left $F$-linearly independent: $s^j$ is
homogeneous of degree $j$ and has nonzero coefficient of $x_1^j$. Denote  by $F[s]$ the subring of $S$ generated by $F$ and $s$. Since
\[
 x_1=(2t)^{-1}(ts+st),\qquad
 x_2=(2t)^{-1}(ts-st),
\] we have $$F[s]=S.$$
The element $s$ is not automorphic over $F$. Otherwise, comparing
the $x_1$ and $x_2$ coefficients in $st=\rho(t)s+\epsilon(t)$ would force
$\rho(t)=t$ and $\rho(t)=-t$ simultaneously.

There is no contradiction with dimension: the pair $x_1,x_2$ is
still independent, and
$\mathcal{L}_F(S)=\atrdeg_F(S)=2$.
What fails is the equality
\[
 F[s]=\sum_{j\geq0}Fs^j.
\]
The right-hand side does not contain $x_1$, as is seen by comparing
homogeneous degrees and then the two degree-one coefficients.
Thus the automorphic requirement controls the structure of the
generated subring, not the minimal number of arbitrary ring
generators over the coefficient division ring.
\end{example}

{


}
\begin{example}[Strict inequalities between the three dimensions]
\label{ex:derivation}
Let $D=\Q(t)$ and let $\varphi:D\to D$ be the injective
endomorphism defined by
\[
 \varphi(r(t))=r(t^2).
\]
Define
\[
 S=\bigoplus_{n\geq0}x^nD
\]
with multiplication determined by
\[
 (x^i a)(x^j b)=x^{i+j}\varphi^j(a)b
 \qquad(a,b\in D,\ i,j\geq0).
\]
Thus $dx=x\varphi(d)$ for every $d\in D$; in particular,
$tx=xt^2$. Equivalently, $S$ is the opposite ring of the
skew polynomial ring defined by the injective, nonsurjective
endomorphism $\varphi$. We claim that
\[
 \atrdeg_D(S)=0,\qquad
 \dL(S)=1,\qquad
 \GKdim_D(S)=\infty.
\]

First, suppose that an element
\[
 f=\sum_{i=0}^{m}x^i a_i\notin D,
 \qquad m\geq1,\quad a_m\ne0,
\]
is automorphic over $D$, with associated pair
$(\rho,\varepsilon)$. Comparing the coefficients of $x^m$ in
$fd=\rho(d)f+\varepsilon(d)$ gives
\[
 a_m d=\varphi^m(\rho(d))a_m
 \qquad(d\in D).
\]
Since $D$ is commutative, this implies
$d=\varphi^m(\rho(d))$ for every $d\in D$. This is impossible,
because
\[
 \varphi^m(D)=\Q(t^{2^m})\subsetneq\Q(t)=D.
\]
Thus every automorphic element of $S$ lies in $D$, and hence
$\atrdeg_D(S)=0$. 

Next, the powers of $x$ are left $D$-linearly independent:
indeed,
\[
 \sum_{i=0}^{N}d_i x^i
 =\sum_{i=0}^{N}x^i\varphi^i(d_i)=0
\]
forces every $d_i$ to vanish. Therefore $\dL(S)\geq1$.
To prove the reverse inequality, we bound the growth of
centralizers over $\Q$.

Fix $f=\sum_{i=0}^{m}x^i a_i\notin D$, with $a_m\ne0$, and put
\[
 C=C_S(f),\qquad
 C_{\leq n}=C\cap\bigoplus_{i=0}^{n}x^iD.
\]
If $g=\sum_{j=0}^{n}x^j b_j\in C_{\leq n}$, comparison of the
coefficients of $x^{m+n}$ in $fg=gf$ yields
\[
 \varphi^n(a_m)b_n=\varphi^m(b_n)a_m.
\]
The ratio of any two nonzero solutions for $b_n$ is fixed by
$\varphi^m$. Moreover,
\[
 \operatorname{Fix}(\varphi^m)=\Q.
\]
Indeed, if a nonconstant rational function $r=p/q$ is written
in lowest terms, substitution $t\mapsto t^{2^m}$ multiplies
$\max\{\deg p,\deg q\}$ by $2^m$, so
$r(t^{2^m})=r(t)$ is impossible.
Consequently, the possible coefficients $b_n$ form a
$\Q$-vector space of dimension at most one. The coefficient
map on $C_{\leq n}$ has kernel $C_{\leq n-1}$, where
$C_{\leq-1}=0$, and hence
\[
 \dim_{\Q}C_{\leq n}\leq n+1.
\]

If $f,g$ were a commuting pair with left $D$-linearly
independent indexed monomials, then $f,g\notin D$.
Set $c=\max\{\deg_x f,\deg_x g\}$. For every $q\geq0$, the
monomials $f^i g^j$ with $i+j\leq q$ would be
$\Q$-linearly independent elements of $C_{\leq cq}$.
This would imply
\[
 \binom{q+2}{2}\leq\dim_{\Q}C_{\leq cq}\leq cq+1,
\]
a contradiction for sufficiently large $q$. Thus
$\dL(S)=1$.

Finally, take the left $D$-frame $V=D+Dx$. For every $n\geq0$,
the left action of $D$ on $x^nD$ is given by
\[
 d(x^n a)=x^n\varphi^n(d)a.
\]
It follows that
\[
 \ldim{x^nD}
 =[D:\varphi^n(D)]
 =[\Q(t):\Q(t^{2^n})]
 =2^n.
\]
Since $x^nD\subseteq V^{n+1}$, we obtain
\[
 \ldim{V^q}\geq2^{q-1}\qquad(q\geq1).
\]
Therefore
\[
 \GKdim_D(S)
 \geq
 \limsup_{q\to\infty}
 \frac{\log\ldim{V^q}}{\log q}
 =\infty.
\]
Each $V^q$ is nevertheless finite-dimensional, since it is
contained in $\bigoplus_{i=0}^{q}x^iD$.
Thus the infinite dimension comes from exponential growth.
In particular, $S/D$ cannot be a compatible Ore quotient in
the sense of Definition~\ref{def:afg}.
\end{example}

\end{document}